\documentclass[11pt]{article}

\usepackage{amsmath,amssymb,amsthm, color}
\usepackage[utf8]{inputenc}

\definecolor{darkgreen}{rgb}{0,0.55,0}

\newtheorem{thm}{Theorem}
\newtheorem{lem}{Lemma}
\newtheorem{cor}{Corollary}
\newtheorem{prop}{Proposition}
\newtheorem{rem}{Remark}

\newcommand{\eps}{\varepsilon}
\newcommand{\ep}{\epsilon}
\newcommand{\R}{\mathbb{R}}
\newcommand{\C}{\mathbb{C}}

\newcommand{\boE}{\mathcal{E}}
\newcommand{\Ee}{{E}_{\varepsilon,\eta}}

\newcommand{\logeps}{|\!\log\eps|}
\newcommand{\rhotf}{\rho_{\scriptscriptstyle{TF}}}
\newcommand{\Otf}{\Omega_{\scriptscriptstyle{TF}}}
\newcommand{\TF}{\scriptscriptstyle{TF}}
\def\rest{\hskip 1pt{\hbox to 10.8pt{\hfill
\vrule height 7pt width 0.4pt depth 0pt\hbox{\vrule height 0.4pt
width 7.6pt depth 0pt}\hfill}}}
\def\evalu{\hskip 1pt{\hbox to 2pt{\hfill \vrule height -6pt width 0.4pt depth
0pt}}}
\def\barint{\mathop{\vrule width 6pt height 3 pt depth -2.5pt \kern -8.8pt
\intop}}

\title{Vortex dynamics for the  two dimensional non homogeneous Gross-Pitaevskii
equation}
\author{{\sc Robert  L. Jerrard}  \& {\sc Didier Smets} }
\date{}
\begin{document}

\maketitle
\begin{abstract}
We derive the asymptotical dynamical law for Ginzburg-Landau
vortices in an inhomogeneous background density under the Schr\"odinger
dynamics, when the Ginzburg-Landau parameter goes to zero. New ingredients
involve across the cores lower bound estimates and approximations. 
\end{abstract}
\section{Introduction}                                          %

We are interested in the two dimensional Gross-Pitaevskii equation
\renewcommand{\theequation}{GP}
\begin{equation}\label{eq:gpv}
i \partial_t u - \Delta u + \frac{1}{\eps^2}\left( V(x) + |u|^2\right) u  = 0
\end{equation}
for $u:\R^2\times \R^+ \to \C$, where $0<\eps\ll1$ and $V : \R^2 \to \R^+$ is a smooth potential
such that
\setcounter{equation}{0}
\renewcommand{\theequation}{\arabic{equation}}
$$
V(x) \to +\infty \text{ as } |x|\to +\infty. 
$$ 
The Gross-Pitaevskii equation is a widely used model to describe the dynamics of a Bose-Einstein 
condensate in a trapping potential $V$. The equation on $\R^2$ arises
via dimension reduction from 3 dimensions; this has been justified for particular choices of $V$
in \cite{AMSW} for example.

Equation \eqref{eq:gpv} is Hamiltonian, with Hamiltonian given by 
$$
\mathcal{E}_{\eps,V}(u) = \int_{\R^2} \frac{|\nabla u|^2}{2} + \frac{1}{\eps^2}\left(
V(x)\frac{|u|^2}{2} + \frac{|u|^4}{4}\right).
$$ 
Another quantity which is preserved by the flow associated with \eqref{eq:gpv}
is the total mass $M$, given by
$$
 M(u) = \int_{\R^2} |u|^2.
$$ 
For each $m>0$, there exists\footnote{We refer to Section \ref{sect:groundstate} for the details 
on a number of statements regarding the ground states which we state without justification in this introduction.}    
at least one positive ground state $\eta \equiv \eta_{\eps,m}:\ \R^2\to \R^+$ of
total mass equal to $m$.  By definition, a ground state $\eta$ realizes the infimum
$$
\mathcal{E}_{\eps,V}(\eta) = \inf\{ \mathcal{E}_{\eps,V}(g), \ g \in H^1(\R^2,\C),\ M(g) = m\},
$$
and satisfies the Euler-Lagrange equation
$$
-\Delta \eta + \frac{1}{\eps^2}(V+\eta^2) \eta = \frac 1{\eps^2}\lambda \eta,
$$ 
where we write the Lagrange multiplier as $\frac{1}{\eps^2}\lambda$ for some $\lambda\equiv \lambda_{\eps,m}.$ 

In the limit $\eps \to 0,$ we have 
$$
\eta^2 \to \rhotf \qquad\text{in } L^2(\R^2),
$$
where the function $\rhotf$, known as the {\it Thomas-Fermi profile} in the physics literature, is given by
$\rhotf(x) := (\lambda_0-V)^+(x)$ where the number $\lambda_0$  is uniquely determined by the mass condition
$$
\int_{\R^2}(\lambda_0-V(x))^+\, dx = m.
$$
We will study the behaviour of solutions of \eqref{eq:gpv}
which correspond, in a sense to be made precise later, to perturbations of the ground state $\eta$ 
by a finite number of quantized vortices, each carrying a single quantum of vorticity. Our goal 
is to prove that these vortices persist, and to describe their evolution in time.

We will show  that to leading order the vortices do not
interact, and that each one evolves (in a renormalized time scale) by the
orthogonal gradient flow for the function $\log \rhotf$,
with a sign depending on the winding number of the given vortex. More precisely, 
let 
$$\Otf := \{ x: \rhotf(x)>0\}$$
be the interior of the limiting support\footnote{Note that we do {\em not} assume that $\Otf$ is simply connected or that
its boundary is smooth.} of the ground state, let $\{b_i^0\}_{i=1}^l$ be distinct points in  $\Otf$, and
let $d_1,\ldots, d_l \in \{-1,+1\}.$ For each $i\in \{1,\cdots,l\},$ we denote
by $b_i(t)$ the solution of the ordinary differential equation
\begin{equation}\label{eq:limitdynamics0}
	\dot b_i (t) = d_i \frac{\nabla^\perp\rhotf}{\rhotf}(b_i(t)),
\end{equation}
where $\nabla^\perp = (-\partial_{x_2}, \partial_{x_1})$,
with initial datum $b_i(0) = b_i^0.$
\begin{thm}\label{thm:limit}
	Let $(u_\eps^0)_{\eps >0}$ be a family of initial data for \eqref{eq:gpv} such that
	$$
	M(u_\eps^0) = m,
	$$
	$$
	\mathcal{E}_{\eps,V}(u_\eps^0) \le \mathcal{E}_{\eps,V}(\eta) + \pi \sum_{i=1}^l \rhotf(b_i^0)\logeps +  o(\logeps),
	$$
	and
	$$
	{\rm curl} \big( \frac{j(u_\eps^0)}{\rhotf}\big) \longrightarrow 2\pi \sum_{i=1}^l d_i \delta_{b_i^0} \qquad \text{in } W^{-1,1}_{\rm loc}(\Otf),
	$$
as $\eps \to 0.$ Then, as long as the points $\{b_i(t)\}_{i=1}^l$ remain distinct,
	$$
	{\rm curl} \big( \frac{j(u_\eps^t)}{\rhotf}\big) \longrightarrow 2\pi \sum_{i=1}^l d_i \delta_{b_i(t)} \qquad \text{in } W^{-1,1}_{\rm loc}(\Otf),
	$$
	as $\eps \to 0,$ where $u_\eps^t:= u_\eps(\cdot,t\logeps).$ 
\end{thm}
Here, $j(u_\eps^t):=(iu_\eps^t,\nabla u_\eps^t)$ where $(z,w):={\rm Im}(z \bar{w}).$ Therefore,
$$ \frac12 {\rm curl} \big( \frac{j(u_\eps^t)}{\rhotf}\big) = \frac12 {\rm curl}\: j\big(\frac{u_\eps^t}{\sqrt{\rhotf}}\big) = J\big(\frac{u_\eps^t}{\sqrt{\rhotf}}\big)$$ 
is the Jacobian determinant of $u_\eps^t/\sqrt{\rhotf}.$ It is widely recognized, in the present regime for the Ginzburg-Landau energy, that the notion of a vortex of winding number $d_i$ located at the point $b_i(t)$ is appropriately 
described by the presence of the term $2\pi d_i \delta_{b_i(t)}$ in the limit of the vorticity field 
${\rm curl}\: j\big(u_\eps^t/\sqrt{\rhotf}\big).$

\begin{rem}{\rm
Note that the ordinary differential equations \eqref{eq:limitdynamics0} are decoupled. Also,  since $\rhotf(b_i(t)) = \rhotf(b_i^0)$ for any $t\in \R$ the points $ \{b_i(t)\}_{i=1}^l$ remain distinct for all times
	unless two of them are located on the same level line of $\rhotf$ and have opposite circulations.
	}
\end{rem}

Results of this sort in the homogeneous case
$\eta\equiv 1$ were first proved in the late 1990s, see \cite{CoJe, CoJe2},
and have subsequently been developed by a number of authors, see for example  
\cite{LiXi,BeJeSm,JeSp}.
The point of this paper is thus to understand the effect of the
inhomogeneity on the dynamical law for the vortices.\\
We remark that a number of authors have studied questions about vortex dynamics in
inhomogeneous backgrounds for parabolic
equations  \cite{Li, JianSong}, or more recently \cite{SerfTice} for 
a quite general class of equations of mixed parabolic-Schr\"odinger 
type. The case of pure Schr\"odinger  dynamics presents distinct  difficulties and as far as we know has not been treated until now.

The sequel of this introduction is devoted to the presentation of the strategy leading to Theorem \ref{thm:limit}. We notice that will actually prove a result (Theorem \ref{thm:main} below) which is stronger in two respects than Theorem \ref{thm:limit}: first it describes the dynamics of vortices at small but fixed value of $\eps$, rather than
asymptotically as $\eps \to 0$ in Theorem \ref{thm:limit}, and second it applies to a broader class of inhomogeneous equations (see \eqref{eq:nhgp} below) where $\eta$ need not necessarily be the profile of a ground state.   


\subsection{Perturbation equation and Theorem \ref{thm:main}}

For the class of initial data which we consider in Theorem \ref{thm:limit}, it is convenient to 
rewrite the corresponding solutions of \eqref{eq:gpv} in the form 
\begin{equation}\label{eq:change}
	u(x,t) = \eta(x)  w(x,t)
\end{equation}  
and to study the evolution equation for $w$.
One easily checks that if $u$ is a solution to \eqref{eq:gpv}, then $w$ solves
$$
i\eta^2 \partial_t w - {\rm div}(\eta^2\nabla w) +\frac{1}{\eps^2} \eta^4
(|w|^2-1)w = -\frac{ \lambda}{\eps^2} \eta^2 w.
$$
In particular, the change of phase and time scale 
$$
v(x,t) =
\exp\left(i\frac{\lambda}{\eps^2}\frac{t}{\logeps}\right)w\left(x,\frac{t}{\logeps}
\right) 
$$
leads to the equation  
\renewcommand{\theequation}{NHG}
\begin{equation}\label{eq:nhgp}
i\logeps \eta^2 \partial_t v - {\rm div}(\eta^2\nabla v) +\frac{1}{\eps^2}
\eta^4 (|v|^2-1)v = 0
\end{equation}
for $v$. Note that the change of time scale is related to the fact that the 
phenomenon which we wish to describe, namely vortex motion, arises in times of order
one in that new time scale (see the definition of $u_\eps^t$ in the statement of Theorem \ref{thm:limit}). 
\renewcommand{\theequation}{\arabic{equation}}
\setcounter{equation}{2}

\medskip

Our analysis will henceforth focus on equation
\eqref{eq:nhgp}. Equation \eqref{eq:nhgp}, like \eqref{eq:gpv},  is Hamiltonian, with Hamiltonian
given  by the weighted Ginzburg-Landau energy
\begin{equation}\label{Eepseta}
E_{\eps,\eta}(v) \equiv \int_{\R^2} e_{\eps,\eta}(v) =   \int_{\R^2} \eta^2
\frac{|\nabla v|^2}{2} + \eta^4 \frac{(|v|^2-1)^2}{4\eps^2}. 
\end{equation}
As a matter of fact, using the Euler-Lagrange equation for $\eta$, one realizes 
that    
\begin{equation}\label{eq:parotide2}
\mathcal{E}_{\eps,V}(u) = \mathcal{E}_{\eps,V}(\eta) + E_{\eps,\eta}(v) + \frac{\lambda}{2\eps^2}\Big( M(u)-M(\eta)\Big).
\end{equation}
In the sequel, we enlarge our framework and consider equation \eqref{eq:nhgp} where $\eta\ : \ \R^2 \to \R$ is any smooth positive function such that the corresponding Cauchy problem is globally well-posed for initial data in $H^1(\R^2,\eta\,dx)$ and such that the corresponding solutions can be approximated by smooth solutions\footnote{This can be verified for a wide 
variety of weight functions $\eta,$ but we wish not consider that discussion here since we already know by means of the change of unknown \eqref{eq:change} that it satisfied when $\eta$ is a ground state.}. In particular, under those assumptions the energy $E_{\eps,\eta}$ is preserved along the flow of \eqref{eq:nhgp}. 

Let $\eps >0$ and let $\Omega \subset \R^2$ be a bounded open set. Let $\{a_i^0\}_{i=1}^l$ be distinct points 
in  $\Omega$, and let $d_1,\ldots, d_l \in \{-1,+1\}.$ For each $i\in \{1,\cdots,l\},$ we denote
by $a_i(t)$ the solution, as long as it does not reach $\partial\Omega,$ of the ordinary differential equation
\begin{equation}\label{eq:limitdynamics}
\dot a_i (t) = d_i \nabla^\perp \log\eta^2(a_i(t)),
\end{equation}
where $\nabla^\perp = (-\partial_{x_2}, \partial_{x_1})$,
with initial datum $a_i(0) = a_i^0.$\\
We assume that
$$
\eta_{min} :=  
\inf_{x\in \Omega} \eta(x) >  0,
$$
and we fix a time $T_{\rm col}>0$ such that
\begin{equation}
	\rho_{min} := \min_{t\in [0,T_{\rm col}]}
\min\big\{\{ \frac 12 d(a_i(t),a_j(t))\}_{\ i\neq j}\cup \{d(a_i(t),\partial \Omega)\}_{i} \cup \{1\} \big\}
 > 0.
\label{rhomin.def}\end{equation}
Finally, we consider a finite energy solution $v$ of \eqref{eq:nhgp}, we set $v^t:=v(\cdot,t)$ and we define, for $t\in [0,T_{\rm col}],$
\begin{equation}\label{eq:defrta}
 r_a^t := \| J v^t - \pi \sum_{i=1}^l d_i
\delta_{a_i(t)}\|_{W^{-1,1}(\Omega)},
\end{equation}
and
\begin{equation}\label{defSigmat}
\Sigma^t := \frac{\Ee(v^t)}{\logeps} - \pi \sum_{i=1}^l \eta^2(a_i(t)).
\end{equation}

We will deduce Theorem \ref{thm:limit} from 

\begin{thm}\label{thm:main}
There exist positive constants $\eps_0,$ $\gamma_0$ and $C_0$, depending
only on $l$, $\eta_{min}$, $\rho_{min}$, and $\|\nabla \eta^2\|_{L^\infty(\Omega)}$, such that if $0<\eps\leq \eps_0$ and if $\Sigma^0 + r_a^0 \leq \gamma_0,$
then 
\begin{equation}\label{eq:ineqmainalt}
	r_a^t \leq r_a^0 + \Big( \Sigma^0 + r_a^0 + \frac{\log\logeps}{\logeps} \Big) \big( e^{C_0t}-1\big) + C_0\eps^\frac12,
\end{equation}
as long as $t\leq T_{\rm col}$ and $\Sigma^0 + r_a^t(t) \leq \gamma_0.$
\end{thm}

\begin{rem}\label{rem:2}{\rm
	$i)$ As we shall discuss in Section \ref{sec:heuristics} below, the
quantity $r^t_a$, which is a sort of discrepancy measure,
can be thought of as measuring the distances between
the ``actual vortex locations'' and  the desired vortex locations.
The quantity $\Sigma^t$, multiplied by $\logeps$, corresponds to the
excess of energy of the solution $v$ with respect to an energy minimizing
field possessing the vortices at the points $a_i(t).$
 Notice that since $\Ee$ 
is preserved by the flow for $v$ and $\eta^2$ is
preserved by the flow
for the $a_i's$, we have
$$
\Sigma^t \equiv \Sigma^0, \qquad \forall \ t\in [0,T].
$$

$ii)$ Theorem \ref{thm:main} is interesting for initial data such that $\Sigma^0 + r^0_a$ is small.
The existence of such data is standard. For example, 
if we fix 
$f:[0,\infty)\to [0,1]$ such that $f'\ge 0$,  $f(0)=0$, and $f(s)\to 1$ as $s\to \infty$,
then for the initial data
\[
	v^0(z) := \prod_{i=1}^l f(\frac{|z-a_i|}\eps) \left( \frac{z- a_i^0}{|z-a_i^0|}\right)^{d_i},
\qquad\qquad\mbox{  $x := (x_1, x_2) \cong z = x_1+ix_2$},
\]
one can check that $\Sigma^0 \le C\logeps^{-1}, r^0_a\le C \eps$. In any case, \eqref{eq:ineqmainalt} contains
the error term in $\log\logeps/\logeps$ which implies that \eqref{eq:ineqmainalt} only yield the inequality 
$\Sigma^0 +r_a^t \leq \gamma_0$ for times at most of order $\log\logeps.$  


$iii)$ One could supplement the claims of Theorem \ref{thm:main}
with closeness estimate for $j(v)$ to a reference field $j_*$ of very simple form. This would follow from an application of Corollary \ref{cor:proche} below; however at the level of approximation which we have adopted here it is 
only meaningful in a neighborhood of size $o(1/\logeps)$ of the vortex core.  
}
\end{rem}

\subsection{Elements in the proofs}\label{sec:heuristics}

Under the conditions that will prevail throughout most of this paper, 
we will be able to identify points $\xi_1^t,\ldots \xi_l^t$ and
a number $r_\xi^t$ such that
\begin{equation}
\| Jv^t - \pi\sum_{i=1}^l d_i \delta_{\xi_i^t}\|_{W^{-1,1}(\Omega)}\ \le r^t_\xi \  \le  \ \eps^{1/2}\ \ll \  r^t_a.
\label{heuristics0}\end{equation}
This is expressed in Proposition \ref{prop:loca} below, and entitles us to think of $\xi_i^t, i=1,\ldots, l$ as being the ``actual locations
of the vortices" in $v^t$, up to precision of order $\le r^t_\xi$. Admitting this interpretation, then basic facts about the $W^{-1,1}$ norm, recalled in Section \ref{sect:w11}, 
imply that 
\begin{equation}
r^t_a = \frac 1 \pi   (1 + o(1))
 \sum_{i=1}^l |\xi_i^t -a_i^t|
\label{heuristics1}
\end{equation}
is essentially the aggregate 
distance between the actual vortex locations and  the desired vortex locations, as remarked above.

Heuristic considerations also suggest that if $v^t$ is a function with vortices
at points $\xi_1^t,\ldots, \xi_l^t$ (or more precisely, if \eqref{heuristics0} holds), then
\begin{equation}\label{heuristics2}
E_{\eps, \eta}(v^t) \gtrapprox \pi \logeps (1-o(1)) \sum_{i=1}^l \eta^2(\xi_i^t).
\end{equation}
Hence $E_{\eps, \eta}(v^t) - \pi \logeps \sum_{i=1}^l \eta^2(\xi_i^t)$
corresponds to energy that is not committed to
the vortices, and this energy in principle can
cause difficulties for our analysis. From \eqref{heuristics0}, \eqref{heuristics1}, 
we have
\begin{equation}\label{heuristics3}\begin{split}
E_{\eps, \eta}(v^t) - \pi \logeps \sum_{i=1}^l \eta^2(\xi_i^t)
&\le 
\logeps \left(\Sigma^t  + \frac 1{\pi}(1+o(1))\| \nabla \eta^2\|_\infty r^t_a\right)\\
&\leq \logeps \left( \Sigma^0  +C r^t_a\right).
\end{split}\end{equation}
For our analysis, it suffices to use estimates in the spirit of \eqref{heuristics3} that are a 
little weaker than those suggested in \eqref{heuristics3}, these
are established in Proposition \ref{prop:relate}. We expect from \eqref{heuristics2} and \eqref{heuristics3}
that control of $r^t_a$ should yield a good deal of information about $v^t$. This is expressed in 
Proposition \ref{prop:approx}, where we compare $j(v^t)$ to a reference field $j_*^t.$ An important feature of that approximation is that it holds across the vortex core. 


In order to control the evolution in time of $r_a^t$, we rely on some evolution equations satisfied by 
smooth solutions of \eqref{eq:nhgp}. 
Conservation of energy is a consequence of the identity
\begin{equation}\label{eq:conserve_e}
\partial_t e_{\eps,\eta}(v) = \ {\rm div}( \eta^2 (\nabla v, v_t)),
\end{equation}
and the canonical equation for conservation of mass can be written
\begin{equation}\label{eq:continuity}
\frac \logeps 2  \partial_t \big (\eta^2(|v|^2-1)\big) = 
\nabla \cdot (\eta^2 j(v)) .
\end{equation}
The vorticity $Jv$ satisfies an evolution equation that it is convenient to write in integral form:
\begin{equation}
\frac{d}{dt} \int_\Omega \varphi Jv = \frac 1 \logeps \int_\Omega 
\left( \ep_{lj} \varphi_{x_l}
\frac{\eta^2_{x_k}}{\eta^2} \left[ v_{x_j}\cdot v_{x_k} + \delta_{jk} \frac{\eta^2}{\eps^2}
(|v|^2-1)^2\right] +  \ep_{lj}  \varphi_{x_kx_l} v_{x_j}\cdot v_{x_k}\right)
\label{eq:evoljac}\end{equation}
where $\varphi$ is any smooth, compactly supported test function and $\eps_{lj}$ is the usual antisymmetric tensor.
This follows from the fact that $Jv = \frac 12 {\rm curl}\: j(v)$
together with the equation for the evolution of $j(v)$,
which is obtained from \eqref{eq:nhgp} after multiplying 
by $\nabla v$ and rewriting the result. 

Identity \eqref{eq:evoljac}
is central to our analysis of vortex dynamics, as in previous works
\cite{CoJe, CoJe2, LiXi,BeJeSm,JeSp} on the homogeneous case
(for which of course \eqref{eq:evoljac} still holds, with $\eta \equiv 1$).
Under the conditions that $Jv$ is approximately a measure of the form
$\pi \sum_{i=1}^l d_i \delta_{\xi_i(t)}$, where $\xi_i(t)$ are the vortex 
locations and $d_i\in \{\pm 1\}$ their signs, one expects  the left-hand side
of \eqref{eq:evoljac} to satisfy
\[
\frac{d}{dt} \int_\Omega \varphi Jv 
\ \approx \ 
\frac{d}{dt} \int_\Omega \varphi (\pi \sum d_i \delta_{\xi_i(t)})
\ \approx  \ 
\frac d{dt}( \pi \sum  d_i \varphi(\xi_i(t)))
\ = \ 
\pi \sum d_i \nabla\varphi(\xi_i(t))\cdot \dot \xi_i(t).
\]
Assuming that this holds, then to understand the vortex velocities $\dot \xi_i$,
it only suffices to understand the right-hand side of \eqref{eq:evoljac}.
It turns out that it also suffices to consider test functions $\varphi$ that are linear near each 
vortex. For such test functions, in the homogeneous case $\nabla \eta^2\equiv 0$, the integrand on the right-hand side of \eqref{eq:evoljac} is supported away from the vortex locations, and
one is thus able to control vortex
dynamics by controlling terms of the form $v_{x_i}\cdot v_{x_j}$ away
from the vortex cores. This argument is a key feature of all existing work on
vortex dynamics in the homogeneous case.

When $\nabla \eta^2\ne 0$, it becomes necessary to 
control terms like $v_{x_i}\cdot v_{x_j}$  across the vortex cores.
Carrying this out, in particular relying on the approximation given by  Proposition \ref{prop:approx}, 
is the main new point in our analysis. Once this is established, the whole argument is completed by using a Gronwall type argument on
a quantity related to $r_a^t$, namely $\Sigma^0+g(r_a^t)$, where the function $g$b is defined in \eqref{eq:defg}. This
demonstrates in particular that the new information found in Proposition \ref{prop:approx} is strong enough to close the estimates and conclude the proof.

\medskip
{\noindent \bf Acknowledgments.} 
This research was partially supported by the National
Science and Engineering Research Council of Canada under operating Grant
261955, as well as by the projects ``Around the dynamics of the Gross-Pitaevskii equation'' (JC09-437086) and ``Schr\"o\-din\-ger equations and applications'' (ANR-12-JS01-0005-01) of the Agence Nationale de la Recherche.

\section{A useful lemma}\label{sect:w11}
 
We frequently use the $W^{-1,1}$ norm. The specific convention we use is in our definition is
\[
\| \mu \|_{W^{-1,1}(\Omega)} := \sup \{ \langle \mu, \varphi \rangle \ : \ \varphi\in W^{1,\infty}_0(\Omega),
\max \{ \| \varphi\|_\infty, \|\nabla \varphi\|_\infty\} \le 1\} .
\]
In this paper, we will only use this norm on measures or more regular objects, although of course it is well-defined
for a somewhat larger class of distributions.

The following lemma, which we will use numerous times, is an easy special case of classical results (see \cite{BCL} for 
example).

\begin{lem}
Suppose that $\Omega$ is an open subset of $\R^n$, and that
$\{ a_i\}_{i=1}^l$ are distinct points in $\Omega$. Define
$\rho_a := \min \{ \{\frac 12 |a_i -a_j| \}_{i\ne j} \cup \{ d(a_i, \partial \Omega )\}_i \cup \{1\}\}$.
Given any points $\{ \xi_i\}_{i=1}^l$ in $\Omega$ and $\{ d_i\}_{i=1}^l \in \{\pm 1\}^l$,
if
\begin{equation}\label{eq:W11.1}
\| \sum_{i=1}^l d_i \delta(a_i - \xi_i) \|_{W^{-1,1}(\Omega)}   \le \frac 14 \rho_a,
\end{equation}
then (after possibly relabelling the points $\{ \xi_i\}_{i=1}^l$), 
\begin{equation}\label{eq:W11.2}
\| \sum_{i=1}^l d_i \delta(a_i - \xi_i) \|_{W^{-1,1}(\Omega)} = \sum_{i=1}^l |a_i - \xi_i|. 
\end{equation}
\label{lem:W-11}\end{lem}

In the remainder of this paper, we will always tacitly assume that under the conditions of the lemma,
the points $\xi_i$ are labelled so that the conclusion holds.

We give the short proof for the reader's convenience. 

\begin{proof}For $i=1,\ldots, l$, define
$\varphi_i(x) :=   d_i( \frac 12 \rho_a - |x-a_i|)^+$.
Then $\max( \|\varphi_i\|_\infty, \|\nabla\varphi_i\|_\infty )  =\max(  \frac 12 \rho_a, 1) =1$, for every $i$, so
\[
\| \sum_{i=1}^l d_i \delta(a_i - \xi_i) \|_{W^{-1,1}(\Omega)} \ge \langle  \sum_{j=1}^l d_j (\delta_{a_j} - \delta_{\xi_j})
, \varphi_i\rangle = \frac  {\rho_a} 2- \sum_j d_id_j (\frac {\rho_a} 2 - 
|\xi_j-a_i|)^+.
\] 
Then \eqref{eq:W11.1} implies that $\{ \xi_j \}_{j=1}^l \cap B(a_i,\rho_a/2)$ is nonempty for every $i$. Since  $\{  B(a_i,\rho_a/
2) \}_{i=1}^l$ are pairwise disjoint, it follows (after possibly reindexing) that $\{ \xi_j \}_{j=1}^l \cap  B(a_i,\rho_a/2) = \{\xi_i\}$ 
for all $i$. Now let $\varphi = \sum_i\varphi_i$. The functions $\{ \varphi_i\}$ have disjoint support, so $\max( \|\varphi\|_\infty, \|\nabla\varphi\|_\infty ) =1$, and thus
$
\| \sum_{i=1}^l d_i \delta(a_i - \xi_i) \|_{W^{-1,1}(\Omega)} \ge  \langle \sum_{i=1}^l d_i (\delta_{a_i} - \delta_{\xi_i})
,  \varphi \rangle=  \sum_{i=1}^l |a_i - \xi_i|$.
On the other hand, if $\psi$ is any compactly supported 
function such that $\max( \|\psi\|_\infty, \|\nabla\psi\|_\infty ) \le 1$, then
\[
\langle \sum_{i=1}^l d_i (\delta_{a_i} - \delta_{\xi_i}),  \psi \rangle\le
 \sum_{i=1}^l |\psi(a_i) - \psi(\xi_i)| \le
\sum_{i=1}^l |a_i-\xi_i|.
\]
Hence $\| \sum_{i=1}^l d_i \delta(a_i - \xi_i) \|_{W^{-1,1}(\Omega)} \le
\sum_{i=1}^l |a_i-\xi_i|$. 
\end{proof}

\section{Relating weighted and unweighted energy}\label{sect:relating}

In this section, we relate the weighted and unweighted energy under some
localization assumptions on the Jacobian. For a measurable subset $A \subset \R^2$ and $v \in \dot H^1(A,\C)$ we set
$$
E_{\eps,\eta}(v;A) := \int_A e_{\eps,\eta}(v) \quad\text{and}\quad E_\eps(v;A):= E_{\eps,1}(v,A).
$$
Define the function
$g$ on $\R^+$ by  
\begin{equation}\label{eq:defg}
g(x) = \left\{ \begin{array}{ll}
x + \frac{|\! \log x|}{\logeps} & \text{if } x>\frac{1}{\logeps}\\
\frac{1+ \log\logeps}{\logeps} & \text{otherwise.}
\end{array}
\right.
\end{equation}
We have

\begin{prop}\label{prop:relate} 
Let $\Omega\subset \R^2$ an open set, $\{a_i\}_{i=1}^l$ distinct points in
$\Omega$, $\{d_i\}_{i=1}^l \in \{\pm 1\},$ and $\eta : \Omega \to \R$ a positive
Lipschitz function such that $\inf_{\Omega}\eta =: \eta_{\min} >0$.   Set $\rho_a :=
\min\{\{ \frac 12 d(a_i,a_j)\}_{\ i\neq j}\cup \{d(a_i,\partial \Omega)\}_{i} \cup \{1\}
\},$ and let $\eps \leq \exp(-\frac{8}{\rho_a})$ and $v\in \dot
H^1(\Omega,\C)$ be such that 
\begin{equation}\label{eq:surplus}
\Sigma_a := \left( \frac{\Ee(v)}{\logeps} - \pi \sum_{i=1}^l \eta^2(a_i) \right)^+
< +\infty.
\end{equation}
Assume also that
\begin{equation}\label{eq:localisee}
r_a := \| J v - \pi \sum_{i=1}^l d_i \delta_{a_i}\|_{W^{-1,1}(\Omega)} \leq
\frac{\rho_a}{8}.
\end{equation}  
Then there exists a constant $C$, depending only on 
$l$, $\|\nabla \eta^2\|_\infty$ and $\eta_{min}$, 
such that 
\begin{equation}\label{eq:usuelle}
\begin{aligned}
\frac{E_{\tilde \eps}(v; B(a_i, R))}{|\!\log\tilde\eps|} 
&\leq 
\pi  + C(\Sigma_a + g(r_a))\quad\quad\mbox{ for }i=1,\ldots, l\\
\frac{E_{\tilde \eps}(v;\Omega\setminus \cup_{i=1}^l B(a_i, R))}{|\!\log\tilde\eps|} 
&\le  
 C(\Sigma_a + g(r_a))
\end{aligned}
\end{equation}
where
$R = 4\max( r_a, \logeps^{-1}) \le \frac{\rho_a}4$ and $\tilde\eps := \frac{\eps}{\eta_{min}}$,
and the function $g$ is defined in \eqref{eq:defg}.
%
%
\end{prop}

\begin{proof}
Let $r\in [r_a,  \frac{\rho_{a}}{8}]$  be a number that will be fixed later. Then the balls
$\{B(a_i,4r)\}_{i=1}^l$ are disjoint and contained in $\Omega.$ Let $i\in
\{1,\cdots,l\},$ by monotonicity of the $W^{-1,1}$ norms with respect to the domain,
we deduce from \eqref{eq:localisee} that
$$
\| Jv - \pi d_i \delta_{a_i}\|_{W^{-1,1}(B(a_i,4r)} \leq r_a \le r.
$$
It follows from the lower bounds estimates of Jerrard \cite{Je} or Sandier
\cite{Sa} that 
\begin{equation}\label{eq:lowerbound}
E_\delta(v,B(a_i,4r)) \geq \pi \log \frac{4r}{\delta} - K_1,
\end{equation}
for every $\delta>0$, where $K_1$ is a universal constant. We next write
\begin{equation}\label{eq:borneinfeta}
\begin{split}
\Ee(v,B(a_i,4r)) &= \int_{B(a_i,4r)} \eta^2(x) \frac{|\nabla v|^2}{2}
+ \eta^4(x) \frac{(|v|^2-1)^2}{4\eps^2}\\
&\geq \int_{B(a_i,4r)}  \eta^2(x) \big[ \frac{|\nabla v|^2}{2} +
\frac{(|v|^2-1)^2}{4\left(\frac{\eps}{\eta_{min}}\right)^2}\big]\\
&\geq \left( \inf_{x\in B(a_i,4r)} \eta^2(x)\right)  E_{\tilde
\eps}(v,B(a_i,4r)).
\end{split}
\end{equation}
Therefore, from \eqref{eq:lowerbound} with the choice $\delta=\tilde\eps$, and noting that $|\log r| \ge \log |\frac{\rho_a}{8}| 
\ge \log 8 \ge 1$, we
obtain
\begin{equation}\label{eq:borneinfetabis}
\Ee(v,B(a_i,4r)) \geq \eta^2(a_i) \pi \logeps - K_2 \left( r\logeps + |\!\log r| \right),
\end{equation}
where $K_2$ depends only on $\|\nabla\eta^2\|_\infty$ and $\eta_{min}.$

On the other hand, we deduce from \eqref{eq:surplus} and
\eqref{eq:borneinfetabis} that
\begin{equation}\label{eq:bornesupeta}
\begin{split}
\Ee(v,B(a_i,4r)) &\leq \Ee(v,\Omega) - \sum_{j\neq i}
\Ee(v,B(a_j,4r))\\
&\leq \pi \eta^2(a_i)\logeps + \Sigma_a\logeps + (l-1)K_2  \left( r\logeps + |\!
\log r|\right).
\end{split}
\end{equation} 
Hence, going back to \eqref{eq:borneinfeta} we obtain
\begin{equation}\label{eq:onballi}
\begin{split}
E_{\tilde \eps}(v,B(a_i,4r)) &\leq \frac{1}{\left( \inf_{x\in B(a_i,4r)}
\eta^2(x)\right)}  \Ee(v,B(a_i,4r))\\
&\leq \pi |\!\log \tilde\eps| + K_3 (\Sigma_a \logeps + r\logeps + |\!\log r|),
\end{split}
\end{equation}
where $K_3$ depends only on $l$, $\|\nabla\eta^2\|_\infty$ and $\eta_{min}.$

Concerning the energy outside the balls $B(a_i,4r)$, we have from
\eqref{eq:surplus} and \eqref{eq:borneinfetabis} 
\begin{equation}\label{eq:out}
\begin{split}
\Ee(v,\Omega \setminus \cup_iB(a_i,4r)) &= \Ee(v,\Omega) - \sum_{i}
\Ee(v,B(a_j,4r))\\
&\leq \Sigma_a\logeps + lK_2  \left( r\logeps + |\! \log r| \right).
\end{split}
\end{equation} 
Hence,
\begin{equation}\label{eq:outsideballs}
\begin{split}
E_{\tilde \eps}(v,\Omega \setminus \cup_i B(a_i,4r)) &\leq \frac{1}{\inf
\eta^2}  \Ee(v,\Omega\setminus \cup_i B(a_i,4r))\\
&\leq K_4 (\Sigma_a \logeps + r\logeps + |\!\log r|),
\end{split}
\end{equation}
where $K_4$ depends only on $l$, $\|\nabla\eta^2\|_\infty$ and $\eta_{min}.$

The function $r\mapsto r + |\log r|/\logeps$ is minimized taking $r := \max(r_a,\frac{1}{\logeps})$, 
in which cas $r\leq \frac{\rho_{a}}{8}$ by assumption on $r_a$ and $\eps.$ The conclusions \eqref{eq:usuelle}
follow with the choice $C := \max(K_3,K_4).$
\end{proof}

\begin{rem}
If we define
$\tilde \Sigma_a :=  \frac{\Ee(v)}{\logeps} - \pi \sum_{i=1}^l \eta^2(a_i)$ ,
then  \eqref{eq:borneinfetabis} implies that 
\[
\tilde \Sigma_a  \ge  \sum_i  \left(\frac { E_{\eps,\eta}(v, B(a_j, 4r))}\logeps - \pi \eta^2(a_i)\right) 
\ge - l K_2(r + \frac {|\log r|}\logeps)
\]
for every $r\in [r_a, \frac {\rho_a}{8}]$. Choosing $r = \max(r_a,\frac{1}{\logeps})$ as above, 
we find that $\tilde \Sigma_a \ge  - l K_2 g(r_a)$. 
In particular, $\Sigma_a  =  (\tilde \Sigma_a)^+ \le \tilde \Sigma_a +  2 l K_2 g(r_a)$.
So all our estimates remain true if we replace  $C(\Sigma_a + g(r_a))$ by
$C(\tilde \Sigma_a + (2lK_2+1) g(r_a)).$ 
\end{rem}

\section{Improved localization for Jacobians}\label{sect:imploc}

In this section, we prove that if the Jacobian of a function $v$ is known
to be sufficiently localized, then, provided the excess energy of $v$ with
respect to the points of localization is not to big, the localisation is
actually potentially much stronger.  A result in the same spirit  was obtained
by Jerrard and Spirn in \cite{JeSp0} for the Ginzburg-Landau functional without a
weight. Our proof here below makes a direct use of Theorem $1.1$ and Theorem $1.2'$ in \cite{JeSp0} by
relating the weighted and unweighted Ginzburg-Landau energies according to
Section \ref{sect:relating}. 

\begin{prop}\label{prop:loca}
Let $\Omega\subset \R^2$ be a bounded, open set, $\{a_i\}_{i=1}^l$ distinct points in
$\Omega$, $\{d_i\}_{i=1}^l \in \{\pm 1\},$ and $\eta : \Omega \to \R$ a positive
Lipschitz function such that $\inf_{\Omega}\eta =: \eta_{\min} >0$.   Set $\rho_a =
\min\big\{ \{ \frac 12d(a_i,a_j)\}_{\ i\neq j}\cup \{d(a_i,\partial \Omega)\}_{i} \cup \{1\}
\big\}.$    Let $\eps \leq \exp(-\frac{8}{\rho_a})$ and let $v\in \dot
H^1(\Omega,\C)$ be such that 
\begin{equation}\label{eq:surplus2}
\Sigma_a := \left( \frac{\Ee(v)}{\logeps} - \pi \sum_{i=1}^l \eta^2(a_i) \right)^+  < +\infty. 
\end{equation}
Also, assume that 
\begin{equation}\label{eq:localmoins}
r_a = \| J v - \pi \sum_{i=1}^l d_i \delta_{a_i}\|_{W^{-1,1}(\Omega)} \leq
\frac{\rho_a}{16}.
\end{equation}  

Then there exists  $C_1 \ge 1$, depending
only on a lower bound for $\rho_a$ and $\eta_{min}$ and on an upper bound for $l$ and $\|\nabla\eta^2\|_\infty$, 
 and for each $i\in\{1,\cdots,l\}$ there exists a point $\xi_i \in B(a_i,2r_a)$, 
such that
\begin{equation}\label{eq:localplus}
\| J v - \pi \sum_{i=1}^l d_i \delta_{\xi_i}\|_{W^{-1,1}(\Omega)}
\leq r_\xi \equiv r_\xi(\Sigma_a,r_a) \equiv \eps \exp(C_1(\Sigma_a +
g(r_a))\logeps).
\end{equation}
where $g$ is defined in Proposition \ref{prop:relate}.
\end{prop}

\begin{rem}
Note that Lemma \ref{lem:W-11} and \eqref{eq:localmoins}, \eqref{eq:localplus} imply that
\begin{equation}\label{eq:aminusxi}
\sum_{i=1}^l |a_i - \xi_i| \le \frac 1\pi(r_a + r_\xi).
\end{equation}
\end{rem}

\begin{rem}
Since $g(r) \ge \frac {\log\logeps}\logeps$ for every $r$, our requirement that $C_1\ge 1$
implies that
\begin{equation}\label{C0gtr1}
r_\xi \ge \eps\logeps.
\end{equation}
\end{rem}

As mentioned, the proof of Proposition \ref{prop:loca} relies very heavily on estimates from \cite{JeSp0}. Following the proof, we discuss some small adjustments we have made in employing these estimates here. Also, from here upon in many places we will denote by $C$ constants whose actual value may change from
line to line but which could eventually be given a common value depending only on $l$, $\rho_{min}$, $\eta_{min}$ and $\|\nabla \eta^2\|_\infty.$

\begin{proof}
Since $\eps  \leq \exp(-\frac{8}{\rho_a})$, our assumptions imply that the hypotheses of Proposition \ref{prop:relate} are verified. Then, since
$B(a_i, \frac {\rho_a}2) \subset B(a_i, R) \cup (\Omega\setminus \cup_{i=1}^l B(a_i,R))$
for any $R<\frac {\rho_a}2$, and recalling that $g(r) \ge \frac { \log\logeps}{\logeps}$ for all $r$,
we deduce from \eqref{eq:usuelle} that
\begin{equation}\label{eq:usualb}
\frac{E_{\tilde \eps}(v; B(a_i, \frac{\rho_a}2 ))}{\!\log(\frac{\rho_a}{2\tilde\eps})} 
\leq 
\frac{E_{\tilde \eps}(v; B(a_i, \frac{\rho_a}2 ))}{|\!\log \tilde\eps |} 
(1+ 2 \frac{ |\log \frac{\rho_a}2|}{|\!\log\tilde\eps|})
\le
\pi  + C(\Sigma_a +g(r_a))
\end{equation}
for $i=1,\ldots, l$, and similarly \eqref{eq:usuelle} implies that 
\begin{equation}\label{eq:usualc}
\frac{E_{\tilde \eps}(v; \Omega \setminus \cup_{i=1}^l B(a_i, \frac{\rho_a}4 ))}{|\!\log\tilde\eps|} \leq  C(\Sigma_a + g(r_a)).
\end{equation}
According to Theorem 1.2' in \cite{JeSp0},
it follows from \eqref{eq:localmoins} and \eqref{eq:usualb} that for every $i\in \{1,\ldots,l\}$, there
exists some $\xi_i\in B(a_i, 2 r_a)$ such that
\begin{equation}\label{eq:findxi}
\| Jv - \pi d_i \delta_{\xi_i} \|_{W^{-1,1}(B(a_i, \frac{\rho_a}2))} \le 
C \,\tilde \eps \exp[C (\Sigma_a + g(r_a))\logeps].
\end{equation}
In addition,  Theorem 1.1 in \cite{JeSp0} implies that
if $V$ is any bounded, open subset of $\Omega$ 
then
\begin{equation}\label{eq:citeJeSp}
\|Jv\|_{W^{-1,1}(V)} 
\ \le \  C \, \tilde\eps \, E_{\tilde\eps}(v, V)  \exp\left( \frac{ E_{\tilde\eps}(v, V)}{\pi}\right) 
\ \le \ 
C \, \tilde\eps \,  \exp\left(  E_{\tilde\eps}(v, V) \right) .
\end{equation}
In particular, this and \eqref{eq:usualc} imply that
\begin{equation}\label{eq:complement}
	\| J v \|_{W^{-1,1}(\Omega\setminus \cup_{i=1}^l B(a_i \frac{ \rho_a}4))} \le 
	C \,\tilde \eps \exp[C (\Sigma_a + g(r_a))\logeps].
\end{equation}
Now for $i\in \{1,\ldots,l\}$, let $\chi_i\in C^\infty_c(B(a_i, \frac {\rho_a}2))$ be functions
such that $\chi_i = 1$ on $B(a_i \frac{\rho_a}4)$, $0\le \chi_i\le 1$, 
and $\|\nabla \chi_i\|_\infty  \le C\rho_a^{-1}$. 
Also, let $\chi_0 = 1-\sum_{i=1}^l\chi_i$. 
Then for any $\varphi\in C^\infty_0(\Omega)$,
such that $\| \varphi\|_{W^{1,\infty}(\Omega)}\le 1$, 
\begin{align*}
\langle \varphi, Jv - \pi \sum_{i=1}^l d_i \delta_{\xi_i} \rangle
&=
\sum_{j=0}^l
\langle \chi_j \varphi, Jv - \pi \sum_{i=1}^l d_i \delta_{\xi_i} \rangle
\\
&= \langle \chi_0 \varphi, Jv \rangle  +
\sum_{i=1}^l \langle \chi_i \varphi, Jv - \pi d_i \delta_{\xi_i} \rangle
\\
&\le \sum_{i=1}^l \|\chi_i \varphi \|_{W^{1,\infty}} C \eps \exp[ C (\Sigma_a+g(r_a))\logeps)
\end{align*}
where we have used \eqref{eq:findxi} for $i=1,\ldots, l$ and \eqref{eq:complement} for $i=0$.
Thus 
\[
\langle \varphi, Jv - \pi \sum_{i=1}^l d_i \delta_{\xi_i} \rangle
\le 
\frac C{\rho_a} \eps \exp[ C (\Sigma_a+g(r_a))\logeps]\\
\le 
\eps \exp[ C_1 (\Sigma_a+g(r_a))\logeps]
\]
for a suitable $C_1$, depending on the lower bound $\rho_0$ for $\rho_a$ as well as $l, \eta_{min}, \|\nabla\eta^2\|_\infty$.
This implies \eqref{eq:localplus}.
\end{proof}

To facilitate comparison between some facts that we have used above
and the precise statements in \cite{JeSp0}, we make the following remarks.

First, we have used some estimates in cruder but simpler forms than they appear in \cite{JeSp0}.
For example, on the right-hand side of \eqref{eq:findxi}, 
we have replaced an expressions of the form $(C + K_0)^2 \exp( \frac {K_0} \pi)$ from
\cite{JeSp0}, where here we take $K_0 = C(\Sigma_a+ g(r_a))\logeps$, by the simpler expression $C \exp( K_0)$.
We have also used the fact that $K_0 = C(\Sigma_a+ g(r_a))\logeps \ge \log\logeps$
to allow us to absorb some lower-order terms from \cite{JeSp0}.

Second, estimates in \cite{JeSp0} are stated in terms of a slightly different norm,
$\| \mu \|_{\dot W^{-1.1}(V)} :=
\sup\{ \langle \mu, \phi \rangle : \phi\in C^\infty_c(V), \|\nabla \phi\|_\infty \le 1 \}$.
This does not cause any problems for us, since clearly $\|\mu\|_{W^{-1,1}(V)} \le \|\mu \|_{\dot W^{-1,1}(V)}$.

Finally, the estimate corresponding to \eqref{eq:citeJeSp} in \cite{JeSp0} is a special case of a more general result, and 
as stated there requires the additional assumption that  $\frac{E_{\tilde\eps}(v, V)}{|\!\log\tilde \eps| }
<\pi$. However,  since $\| J v\|_{W^{-1,1}(V)}\le \|Jv\|_{L^1(V)} \le 2 E_{\tilde \eps}(v; V)$,
it is clear that \eqref{eq:citeJeSp} is still true if $\frac{E_{\tilde\eps}(v, V)}{|\!\log\tilde \eps| }
\ge \pi$.

\begin{rem} If $\Omega$ is simply connected, then we can alternately argue by
citing a result  from \cite{JeSp} to obtain an estimate of the form
\eqref{eq:localplus} with $C_1$ {\em independent of $\rho_a$}, at the rather small 
expense of having to replace $\frac {\rho_a}{16}$ on the right-hand side of \eqref{eq:localmoins}
by some smaller quantity depending on $l$ as well as $\rho_a$. This is in principle useful if one wants to consider large numbers of vortices. The relevant result (Theorem 3)
from \cite{JeSp} is proved using facts from \cite{JeSp0}, as in the proof above, 
but combining estimates
on the balls and away from the balls in a more  careful way, to  avoid introducing 
the factors of $\rho_a^{-1}$ that arise from the cutoff functions that we have employed here.

The proof of Theorem 3 from \cite{JeSp} can surely be adapted to yield a similar result without the assumption that $\Omega$ be 
simply connected, but since the proof is slightly complicated, we prefer not to tinker with it here.
\end{rem}

\section{Across the core approximation by reference field}\label{sect:jstar}
In this section we prove
\begin{prop}\label{prop:approx}
Let $\Omega\subset \R^2$ be an open set, $\{\xi_i\}_{i=1}^l$ distinct points in
$\Omega$, $\{d_i\}_{i=1}^l \in \{\pm 1\},$ and $\eta : \Omega \to \R$ a positive
Lipschitz function such that $\inf_{\Omega}\eta =: \eta_{\min} >0.$  Set $\rho_\xi =
\min\big\{\{\frac 12 d(
\xi_i,\xi_j)\}_{i\neq j}\cup \{d(\xi_i,\partial \Omega)\}_{i} \cup \{1\}\big\}.$ 
  Let $\eps \leq \exp(-\frac{8}{\rho_\xi})$ and let $v\in \dot
H^1(\Omega,\C)$ be such that
\begin{equation}\label{eq:surplusbis}
\Sigma_\xi = \left( \frac{\Ee(v)}{\logeps} - \pi \sum_{i=1}^l \eta^2(\xi_i) \right)^+ <
+\infty
\end{equation}
and
\begin{equation}\label{eq:localiseebis}
\| J v - \pi \sum_{i=1}^l d_i \delta_{\xi_i}\|_{W^{-1,1}(\Omega)}
\leq r_\xi \equiv \eps\exp(K\logeps)  = \eps^{1-K}  
\end{equation}     
for some $K\leq \frac{1}{2}.$ 

Define $j_* = j_*(\{\xi_i\}, r_\xi)$ in $\Omega$ by 
$$
j_*(x)  =\left\{\begin{array}{ll}
d_i \frac{(x-\xi_i)^\perp}{\max (r_\xi, |x-\xi _i|)^2 } &\text{if } x \in B(\xi_i,\frac{1}{\logeps})\\
0 &\text{if } x \in \Omega \setminus \cup_{i=1}^l B(\xi_i,\frac{1}{\logeps})
\end{array}\right. 
$$
where  $(y_1, y_2)^\perp := (-y_2, y_1)$.
Then
\begin{equation}\label{eq:estimparsurplus}
E_{\eps,\eta}(|v|) + \frac{1}{2}\int_\Omega \eta^2 \left|
\frac{j(v)}{|v|}- j_*\right|^2 \leq 
   (C \Sigma_\xi + K)\logeps +
C\log\logeps, 
\end{equation}
and
\begin{equation}\label{eq:JminusJstar}
\| \nabla\times (j(v) - j_*) \|_{W^{-1,1}} \le C_3 r_\xi.
\end{equation}
where the constant $C$ depends only on $l,$ $\|\nabla \eta^2\|_\infty$ and  
$\eta_{min}.$
\end{prop}

Since $K\le \frac 12$, the assumption that $\eps< \exp(- 8/\rho_\xi)$ implies that
$r_\xi < \frac 1\logeps < \frac {\rho_\xi}8$. In particular, the
balls $B(\xi_i, \logeps^{-1})$, $i=1,\ldots, l$ are pairwise disjoint and contained in $\Omega$.

\begin{proof}
We will use more than once the fact that
\begin{equation}\label{gradsplit}
|\nabla v|^2 = |\nabla |v|\,|^2 + \left| \frac {j(v)|}{|v|}\right|^2.
\end{equation}

{\bf Step 1: verification of \eqref{eq:JminusJstar}}.
A direct calculation, using the definition of $j_*$, shows that for any smooth $\varphi$, 
\[
\langle \varphi , \nabla\times j_* - 2 \pi \sum d_i \delta_{\xi_i} \rangle
= \sum_{i=1}^l \frac 2{r_\xi^2} \int_{B(\xi_i, r_\xi)} d_i(\varphi(x)- \varphi(\xi_i)) \ \le \,C\, l \|\nabla \varphi\|_\infty r_\xi.
\]
Thus $\| \nabla\times j_* - 2 \pi \sum d_i \delta_{\xi_i} \|_{W^{-1,1}(\Omega)} \le C r_\xi$.
Recalling that $Jv = \frac 12 \nabla \times j(v)$, we deduce \eqref{eq:JminusJstar}
from this estimate and our assumption \eqref{eq:localiseebis}.

It remains to prove \eqref{eq:estimparsurplus}.

{\bf Step 2: decomposing the energy.}
Note that our assumptions \eqref{eq:surplusbis}, \eqref{eq:localiseebis} about the points
$\{ \xi_i\}_{i=1}^l$ are exactly the same as the hypotheses \eqref{eq:surplus}, \eqref{eq:localisee}
about the points $\{ a_i\}_{i=1}^l$ in Proposition \ref{prop:relate}, except that here we impose
an additional smallness condition on $r_\xi$. Thus estimates from Proposition \ref{prop:relate}
are all available here. In particular, recalling \eqref{eq:out} with the choice $r = \max(r_\xi, \frac 1\logeps) = \frac 1\logeps$, we see that
\[
E_{\eps,\eta}(v, \Omega\setminus \cup_i B(\xi_i, \frac 4 \logeps)  ) \le 
C(\Sigma_\xi \logeps +  \log\logeps).
\]
In view of \eqref{gradsplit}, and noting that $j_*$ is supported in $\cup_i B(\xi_i, \frac 1 \logeps)$  
to prove \eqref{eq:estimparsurplus} it therefore suffices to show that
\begin{equation}
E_{\eps,\eta}\bigl(|v|, B(\xi_i, \frac 4\logeps)\bigr)+ 
\frac 12 \int_{B(\xi_i, \frac 4\logeps)}\eta^2  \left| \frac{j(v)}{|v|} - j_*\right|^2 
\ \le \ 
(C \Sigma_\xi + K)\logeps +
C\log\logeps
\label{eq:jstar.reduce}\end{equation}
for $i=1,\ldots, l$.
Toward this end, note that
\begin{align*}
 \logeps\left[\pi \eta^2(\xi_i) +\Sigma_\xi + C \frac{ \log\logeps}\logeps \right]
&\overset{\eqref{eq:bornesupeta}} 
\ge 
E_{\eps,\eta}(v, B(\xi_i, \frac 4\logeps) )\\
&\overset{\eqref{gradsplit}}=
E_{\eps,\eta}\bigl(|v|, B(\xi_i, \frac 4\logeps)\bigr)+ 
\frac 12 \int_{B(\xi_i, \frac 4\logeps)}\!\!\!\!\eta^2  \left| \frac{j(v)}{|v|} - j_*\right|^2 
\\&\quad\quad
+\frac 12 \int_{B(\xi_i, \frac 4\logeps)}\eta^2  \left|  j_*\right|^2
+\int_{B(\xi_i, \frac 4\logeps)}\!\!\!\!\! \eta^2 (\frac {j(v)}{|v|} - j_*)\cdot j_*.
\end{align*}
Using the explicit form of $j_*$ and of $r_\xi$,
\begin{align*}
\frac 12 \int_{B(\xi_i, \frac 4\logeps)}\eta^2  \left|  j_*\right|^2
&\ge \frac 12 \ \min_{B(\xi_i, 4\logeps^{-1})}\eta^2 \ 
\int_{B(\xi_i, \frac 4\logeps)}  \left|  j_*\right|^2 \\
&\ge 
\bigl( \eta^2(\xi_i) - C\|\nabla\eta^2\|_\infty \logeps^{-1} \bigr) 
\pi \log \frac {\logeps^{-1}}{r_\xi}\\
&= 
\bigl(\pi \eta^2(\xi_i) - C\frac {\log\logeps}\logeps \bigr) \logeps (1-K).
\end{align*}
By combining the previous two inequalities and rearranging, we see that to prove \eqref{eq:jstar.reduce},
it suffices to check that
\begin{equation}
\left|
\int_{B(\xi_i, \frac 4\logeps)} \eta^2 (\frac {j(v)}{|v|} - j_*)\cdot j_*
\right| \le C \bigl( \Sigma_\xi \logeps +
\log\logeps\bigr).
\label{eq:jstar.reduce2}\end{equation}

{\bf Step 3: proof of \eqref{eq:jstar.reduce2}.}

First note that
\begin{align*}
\int_{B(\xi_i, \frac 4\logeps)} \eta^2 (\frac {j(v)}{|v|} - j_*)\cdot j_*
&=
\int_{B(\xi_i, \frac 4\logeps)}\bigl(\eta^2(x) - \eta^2(\xi_i)\bigr)(\frac {j(v)}{|v|} - j_*)\cdot j_*
\\ 
& \quad
+
\eta^2(\xi_i) \int_{B(\xi_i, \frac 4\logeps)}
 \frac {j(v)}{|v|} \cdot j_*\ 
 \left(1-   |v| \right)
\\
& \quad
+
\eta^2(\xi_i) \int_{B(\xi_i, \frac 4\logeps)}
 ({j(v)} - j_*) \cdot j_*\ . 
\end{align*}
We estimate the three terms on the right-hand side in turn. First,
\begin{align*}
|\int_{B(\xi_i, \frac 4\logeps)}\bigl(\eta^2(x) - \eta^2(\xi_i)\bigr)(\frac {j(v)}{|v|} - j_*)\cdot j_*| \ 
& \le \ 
\frac C \logeps \| \nabla \eta^2\|_\infty ( \| \nabla v\|_2^2 + \| j_*\|_2^2)
\\
& \le \  C\left[  \frac{E_{\eps,\eta}(v) }\logeps+ (1- K +  \logeps^{-1}) \right]\\
\\
&\le
C(\Sigma_\xi + \log\logeps),
\end{align*}
where we have used the fact that
$\logeps^{-1}E_{\eps,\eta}(v) 
\overset{\eqref{eq:surplusbis}}\le C(\Sigma_\xi + l \pi \|\eta^2\|_\infty) \le C(\Sigma_\xi + \log\logeps)$.

Next,
\begin{align*}
\eta^2(\xi_i) |\int_{B(\xi_i, \frac 4\logeps)}
 \frac {j(v)}{|v|} \cdot j_*\  \left(1-   |v| \right)|
&\le 
C  \|j_*\|_\infty \int_{B(\xi_i, \frac 4\logeps)} \frac \eps 2 |\nabla v|^2 + \frac 1 {2\eps} (|v|^2-1)^2\\
&\le 
 C r_\xi^{-1}\eps E_{\eps,\eta}(v) 
\\
&\le
C(\Sigma_\xi + \log\logeps),
\end{align*}
using the lower bound \eqref{C0gtr1} for $r_\xi$ and arguing as above.

To estimate the final term, note that $j_* = \nabla^\perp h$, for 
\[
h(x):= \begin{cases}
d_i \big[ \frac 1{2r_\xi^2}( |x-\xi_i|^2 - 1)  +\log(r_\xi \logeps)\big] 
&\mbox{ if }|x-\xi_i|\le r_\xi\\
d_i \log (|x-\xi_i| \,\logeps)&\mbox{ if }r_\xi \le |x-\xi_i| \le \logeps^{-1}\\[3pt]
0&\mbox{ if }x\not\in \cup_i B(\xi_i, \logeps^{-1}).
\end{cases}
\]
Thus, we can integrate by parts to find that
\begin{align*}
|\int_{B(\xi_i, \frac 4\logeps)}
({j(v)} - j_*) \cdot j_*\ |
&=
|\int_{B(\xi_i, \frac 4\logeps)}
h\ \nabla \times ({j(v)} - j_*)  | 
\\
&\le\max( \|h \|_{\infty} , \|\nabla  h\|_\infty )\ \|  {j(v)} - j_* \|_{W^{-1,1}} \\
&\le C,
\end{align*}
after using \eqref{eq:JminusJstar} and noting that
$\|\nabla h\|_\infty = \|j\|_\infty = r_\xi^{-1}$. This completes the proof. 
\end{proof}

It follows from
the definitions \eqref{eq:surplus2} and \eqref{eq:surplusbis} of $\Sigma_\xi$ and $\Sigma_a$, together with \eqref{eq:aminusxi},
that
\[
\Sigma_\xi \le \Sigma_a + C g(r_a)  
\]
for $C$ depending only on $\|\nabla \eta^2\|_\infty$.
Combining this with Propositions \ref{prop:loca} and  \ref{prop:approx},  we immediately obtain

\begin{cor}\label{cor:proche}Under the assumptions of Proposition \ref{prop:loca}
\begin{equation}\label{eq:estimparsurplusbis}
E_{\eps,\eta}(|v|) + \frac{1}{2}\int_\Omega \eta^2 \left|
\frac{j(v)}{|v|}- j_*\right|^2 \leq C (\Sigma_a +g(r_a))\logeps, 
\end{equation}
where $j_*=j_*(\{\xi_i\},r_\xi)$, the points $\{\xi_i\}_{i=1}^l$ are given by Proposition \ref{prop:loca},
and the constant $C$ depends only on $l,$ $\rho_a$, $\|\nabla \eta^2\|_\infty$ and $\eta_{\min}.$
\end{cor}

\section{Small time upper bound on the speed of vortices}

Let $C_1$ be the constant given by Proposition \ref{prop:loca}
corresponding to the lower bound $\rho_{min}$ (as defined in \eqref{rhomin.def}) for $\rho_a.$. Let also $\eps \leq \exp(-\frac{8}{\rho_{min}}).$ Then the conclusions of
Proposition \ref{prop:loca}, applied to $v^t$ with this choice of constants, are available to us for all
$0 \le t \le T_{col}$. 
Since the conclusions of Proposition \ref{prop:loca}
remain true if we 
increase $C_1$, we may assume that
\begin{equation}\label{eq:gullit}
 \frac{1}{C_1} 
\leq \frac{\rho_{min}}{8},
\end{equation}
which we do in the sequel. We  define the stopping time
$$
T_{loc} = \sup \{t\leq T_{col}\ ;\ \Sigma^0+g(r_a^s) \leq \frac{1}{2C_1},\
\forall\: 0\leq s \leq t\}.
$$
Since the function $g$ satisfies  $g(r)\geq r$ on $\R^+$, for
$t\leq T_{loc}$ we have $r_a^t \leq \frac{1}{2C_1}\leq  \frac{\rho_{min}}{16}$. In
particular, we may apply Proposition \ref{prop:loca} to $v^t$, 
$\{a_i(t)\}_{i=1}^l$ and $\{d_i\}_{i=1}^l$, which yields points $\{\xi_i(t)\}$
such that
\begin{equation}\label{rxit.def}
\| J v^t - \pi \sum_{i=1}^l d_i \delta_{\xi_i(t)}\|_{W^{-1,1}(\Omega)}
\leq r_\xi^t \equiv r_\xi(\Sigma_a^t,r_a^t) \equiv \eps \exp(C_1(\Sigma_a^t +
g(r_a^t))\logeps),
\end{equation}
where\footnote{Proposition \ref{prop:loca} actually uses a version of surplus
energy for which the weighted energy $\Ee$ is restricted to $\Omega.$
Since the
energy density and the weight are non-negative, our definition of
surplus $\Sigma_a^t$ here, integrating on the whole $\R^2$, yields a larger
number, and is therefore compatible with the claim of the proposition.}
$$   
\Sigma_a^t =  \left(\frac{\Ee(v^t)}{\logeps} - \pi \sum_{i=1}^l \eta^2(a_i(t)) \right)^+.
$$
Since $t\mapsto v^t|_\Omega$ is continuous in $H^1(\Omega)$, it is
clear that
$t\mapsto Jv^t$ is continuous as a function from $\R$ into $W^{-1,1}(\Omega)$,
and hence we can choose $\{ \xi_i(t)\}$ to be piecewise constant, and in particular measurable,
as functions of $t$. 
Since $\Ee$ is preserved by the flow for $v$ and $\eta^2(a_i)$ is preserved by
the flow
for the $a_i$'s, we have
$$
\Sigma_a^t \equiv \Sigma^0.
$$
Note in particular that $r_\xi^t \le \sqrt\eps$ for $t<T_{loc}$.

\begin{prop}\label{prop:jacspeed}
There exist positive constants $\tau_0, \eps_0$ and $C$, depending only on $l$, $\rho_{min}$,
$\eta_{min}$ and $\|\nabla \eta^2\|_\infty$,  
such that $\eps_0\leq \exp(-\frac{8}{\rho_{min}})$ and if $0<\eps<\eps_0$ and 
$$\Sigma^0 + g(r_a^t) \leq \frac{1}{4C_1}$$
for some $t\leq T_{loc},$
then $T_{loc} \geq t + \tau_0$ and
\begin{align}
	&\|Jv^s- Jv^t\|_{W^{-1,1}(\Omega)} 
\leq 
C\bigl(|t-s| + 
r_\xi^t\bigr), \label{small_t.1}
\\
&r_\xi^s \leq  
r^t_\xi  + C \logeps \eps^{1/2} \bigl(|s-t| + r_\xi^t\bigr),
\label{small_t.2}\\
&\{ a_i(s), \xi_i(s) \} \subset B(a_i(t), \frac{\rho_{min}}{4}),\qquad i=1,\ldots,l
\label{small_t.3}\end{align}
for every $t\leq s\leq t+\tau_0$. 
\end{prop}

\begin{proof}
For the ease of notation in the present proof,  $\|\cdot\|$
is understood to mean $W^{-1,1}(\Omega)$ while $|\cdot|$  denotes the Euclidean norm
on $\R^2.$

{\bf Step 1}. 
Let $t\leq s \leq  \min\{ T_{loc}, t+ \tau_0\}$, for $\tau_0$ to be fixed below. We first use the fact that $Jv^s, Jv^t$ 
are well-approximated by sums of point masses
to show that $\| Jv^s- Jv^t\|$ can be estimated by computing 
$\langle Jv^s - Jv^t, \varphi\rangle$ for a specific test function $\varphi$ with
certain good properties (in particular, bounds on {\em second} derivatives of $\varphi$).
Toward this end, note that
\begin{equation}\label{eq:decompjac}\begin{split}
\|Jv^s-Jv^t\| &\leq
\|Jv^s-\pi\sum_{i=1}^l d_i\delta_{\xi_i(s)}\|
+
\|Jv^t-\pi\sum_{i=1}^l d_i\delta_{\xi_i(t)}\|
+
\|\pi\sum_{i=1}^l d_i(\delta_{\xi_i(s)}-\delta_{\xi_i(t)})\|\\
&\leq r_\xi^s+r_\xi^t + \pi \sum_{i=1}^l
|\xi_i(s)-\xi_i(t)|.
\end{split}\end{equation}
We now fix $\tau_0$, depending only on $\|\nabla
 \eta^2\|_{\infty}$, $\eta_{min}$  and $\rho_{min}$, such that if $t\leq s\leq t+\tau_0,$ we
have $|a_i(s)-a_i(t)| \leq \frac{\rho_{min}}{8}$ for all $i\in \{1,\cdots
,l\}.$ By Proposition \ref{prop:loca}, the choice of $T_{loc}$, and Lemma \ref{lem:W-11},  for
every $\tau\leq T_{loc}$ we have
$|a_i(\tau)-\xi_i(\tau)|\leq 2 r_a^\tau  \leq
\frac{\rho_{min}}{8}.$
By the triangle inequality, it follows that
$\xi_i(s)\in B(a_i(t),
\frac{\rho_{min}}{4})$
for all $t\leq s \leq \min(t+\tau_0,T_{loc})$  and  $i\in \{1,\cdots
,l\}.$ Let 
\begin{equation}\label{eq:defvarphi}
\varphi(x) = \sum_{i=1}^l d_i
\frac{(x-a_i(t))\cdot(\xi_i(s)-\xi_i(t))}{|\xi_i(s)-\xi_i(t)|} \chi\Big(
|x-a_i(t)|\Big),
\end{equation}
where $\chi\in\mathcal{C}^\infty(\R^+,[0,1])$ is such that 
$\chi\equiv 1$ on $[0,\rho_{min}/4],$ 
$\chi\equiv 0$ on  $[\frac{\rho_{min}}{2},+\infty).$ 
By construction and the definition of $\rho_{min}$, we have $\varphi \in
\mathcal{D}(\Omega)$ and it follows that 
$$
\begin{aligned}
\pi \sum_{i=1}^l |\xi_i(s)-\xi_i(t)| 
&= \langle \pi\sum_{i=1}^l
d_i(\delta_{\xi_i(s)}-\delta_{\xi_i(t)}), \varphi \rangle\\
&\le (r_\xi^t + r_\xi^s)\|\varphi\|_{W^{1,\infty}} + 
\langle Jv^s - J v^t ,\varphi\rangle.
\end{aligned}
$$
Combining this with \eqref{eq:decompjac}, we conclude that 
\begin{equation}\label{eq:forlan}
\mbox{ $\| Jv^s - Jv^t\| \le C( r^s_\xi + r^t_\xi)  +\langle Jv^s - J v^t ,\varphi\rangle$,
\quad\quad\quad
with $\|\varphi\|_{W^{2,\infty}} \le C \rho_{min}^{-2}$.}
\end{equation}

{\bf Step 2}.
We now deduce from \eqref{eq:forlan}, together with
\eqref{eq:evoljac}, the fact that $\Sigma^0 \le \frac 1{4C_1}$, and conservation of energy, that
\begin{equation}\label{eq:vanbasten}\begin{aligned}
\| Jv^s - Jv^t\|&\leq (r_\xi^t + r_\xi^s)\|\varphi\|_{W^{1,\infty}} + (s-t)
\sup_{\tau\in [t,s]}\|\frac{d}{d\tau}
Jv^\tau\|_{W^{-2,1}(\Omega)} \|\varphi\|_{W^{2,\infty}}\\
&\leq C(r_\xi^t + r_\xi^s + |t-s|),
\end{aligned}\end{equation}
for $t\leq s \leq \min(t+\tau_0,T_{loc})$, where $C$ depends only on $l$, $\rho_{min}$,
$\eta_{min}$ and $\|\nabla \eta^2\|_\infty$.

{\bf Step 3}.
It remains to estimate $r_\xi^s$ and to show that $t+\tau_0\leq T_{loc}.$
For that purpose, since $r_\xi^s=r_\xi(\Sigma^0,r_a^s)$, we first use \eqref{eq:vanbasten} to compute
\begin{equation}\label{eq:rxis}\begin{split}
r_a^s &= \| J v^s - \pi \sum_{i=1}^l d_i\delta_{a_i(s)}\| \\ 
&\leq \|Jv^s - J v^t\| + \|Jv^t -\pi\sum_{i=1}^l
d_i\delta_{a_i(t)}\| + \|\pi\sum_{i=1}^l
d_i(\delta_{a_i(s)}-\delta_{a_i(t)})\|\\
&\leq r_a^t + C\bigl( |t-s| + r_\xi^t + r_\xi^s) \bigr).
\end{split}\end{equation} 
Next, since $s\leq T_{loc}$, 
\begin{equation}\label{eq:rxis2}\begin{split}
r_\xi^s &= \eps \exp(C_1[\Sigma^0+g(r_a^s)]\logeps)\\
&\leq r_\xi^t + C_1\logeps\eps^{\frac{1}{2}}\|g'\|_\infty \ (r_a^s-r_a^t)^+\\
&\leq r_\xi^t + \frac{1}{2C}(r_a^s-r_a^t)^+,
\end{split}\end{equation} 
provided we assume, and this is again no loss of generality, that
$C|\!\log \eps_0|\eps_0^{\frac{1}{2}} \leq \frac{1}{2C}$ 
for the same constant  $C$ as in \eqref{eq:rxis}.
Combining \eqref{eq:rxis} with \eqref{eq:rxis2}
we obtain
\begin{equation}\label{eq:rxis3}
r_a^s - r_a^t  \leq  C\bigl( |t-s| + r_\xi^t\bigr). 
\end{equation}
Going back to \eqref{eq:rxis2}, this yields the desired estimate of $r^s_\xi$\, :
$$
r_\xi^s \leq  r^t_\xi  + C \logeps \eps^{1/2} \bigl(|s-t| + r_\xi^t\bigr).
$$
Then going back to \eqref{eq:vanbasten},
$$
\|Jv^s-Jv^t\| \leq C(|t-s| + r_\xi^t)
$$
for $t\leq s \leq \min(t+\tau_0,T_{loc}).$ 
Finally, by assumption we have $\Sigma^0+g(r_a^t) \leq 1/(4C_1)$ so that by
\eqref{eq:rxis3} and the fact that $g'\leq 1$,
$$
\Sigma^0 + g(r_a^s) \leq 1/(4C_1) + C\bigl( |t-s| + r_\xi^t\bigr) \leq 1/(4C_1)
+ C\bigl(\tau_0+\eps^\frac{1}{2}\bigr) \leq 1/(3C_1),
$$
provided we assume, and this is no loss of generality, that
$C(\tau_0+\eps_0^\frac{1}{2}) \leq 1/(12C_1).$ It follows that
$\min(t+\tau_0,T_{loc}) = t+\tau_0$, and the proof is complete. 
\end{proof}

\section{Control of the discrepancy}

In this section, we prove a discrete differential inequality for the quantity
$r_a^t.$
More precisely, we will prove

\begin{prop}\label{prop:discrineq}
There exist  positive constants $\eps_0$ and $C_0$, depending only on $l$, $\rho_{min}$, $\eta_{min}$ and
$\|\nabla \eta^2\|_\infty$,  
such that $\eps_0 \leq \exp(-\frac{8}{\rho_{min}})$ and if $0<\eps<\eps_0$ and 
\begin{equation}
\Sigma^0 + g(r_a^t) \leq \frac{1}{4C_1}
\label{rtasmall}\end{equation}
for some $t\leq T_{loc},$
then
$$
\frac{r_a^T  -r_a^t}{T-t} \leq C_0(\Sigma^0 + g(r_a^t))
$$  
where  
$T = t+ \frac{(r^t_\xi)^2}\eps \leq T_{loc}$.
\end{prop}

This is the main estimate in the proof of Theorem \ref{thm:main}.

\begin{proof}
	We first require the constant $\eps_0$ to be smaller than the one appearing in the statement of Proposition \ref{prop:jacspeed}. As in the proof of Proposition \ref{prop:jacspeed}, we will write simply $\| \cdot \|$
to denote the $W^{-1,1}(\Omega)$ norm.
Note that the condition  \eqref{rtasmall} states exactly that 
\begin{equation}
r^\xi_t \le \eps^{3/4}
\label{eps14}
\end{equation}
and then the definition of $T$ and  \eqref{small_t.2}  yield
\begin{equation}
r^s_\xi \le 2 r^t_\xi  \quad\quad\mbox{ for all  }s\in [t,T]
\label{rsxi}\end{equation}
if $C$ is large enough and $\eps_0$ small enough, which we henceforth take to be the case. Moreover, from \eqref{eq:rxis3}, we see that
$r_a^s \le r_a^t + C(T-t +  r_\xi^t)$ for all $s\in [t,T]$, and then the
choice of $T$ and the
definition of $g$ imply that
\begin{equation}
g(r_a^s) \le 2 g(r_a^t)\quad\quad\mbox{ for all }s\in [t,T].
\label{gras}\end{equation}

{\bf 1}. First note that
\begin{align}
r^T_a - r^t_a 
&= 
\| Jv^T - \pi \sum_{i=1}^l  d_i \delta_{a_i(T)}\|
-
\| Jv^t - \pi \sum_{i=1}^l  d_i \delta_{a_i(t)}\|\nonumber\\
&\le 
\pi \sum_{i=1}^l \left(|\xi_i(T)-a_i(T)| - |\xi_i(t) - a_i(t)| \right)\ + \ r^T_\xi + r^t_\xi
\nonumber\\
&\le 
\pi \sum_{i=1}^l  \nu_i \cdot \bigl(\xi_i(T)- \xi_i(t)  + a_i(t)- a_i(T) \bigr)\ + \ r^T_\xi + r^t_\xi
\label{eq:p5.1}
\end{align}
for $\nu_i = \frac {\xi_i(T)-a_i(T)}{|\xi_i(T)-a_i(T)|}$ (unless $\xi_i(T)- a_i(T)=0$, in which
case $\nu_i$ can be any unit vector).
We now define
\[
\varphi(x) = \sum_i  d_i\nu_i \cdot(x - a_i(t)) \chi(|x - a_i(t)|)
\]
for $\chi\in C^\infty(\R^+, [0,1])$ such that $\chi \equiv 1$ on $[0,\frac 12 \rho_{min}]$ and $\chi\equiv 0$ on $(\rho_
{min},\infty)$. 
It follows from \eqref{small_t.3} that (since $d_i{}^2=1$ for all $i$)
\[
\pi \sum_{i=1}^l  \nu_i \cdot \bigl(\xi_i(T)- \xi_i(t)  + a_i(t)- a_i(T) \bigr)\
=
\pi \sum_{i=1}^l  d_i\Bigl[ \varphi(\xi_i(T))- \varphi(\xi_i(t))  - \varphi(a_i(T)) + \varphi(a_i(t)) \Bigr],
\]
so that \eqref{eq:p5.1} and the definition of $r^T_\xi$ imply that 
\begin{equation}\label{eq:p5.2}
r^T_a-r^t_a \le \langle  \varphi , Jv^T - Jv^t \rangle  \ -
\ \pi\sum_{i=1}^l d_i \Bigl[\varphi(a_i(T)) - \varphi(a_i(t))\Bigr]
+\ C(r^T_\xi + r^t_\xi).
\end{equation}

{\bf 2}.
The remainder of the proof is devoted to an estimate of 
$ \langle  \varphi , Jv^T - Jv^t \rangle $.
First,  using \eqref{eq:evoljac},
\begin{equation}\label{eq:p5.3}
\begin{aligned}
\langle \varphi , Jv^T - Jv^t \rangle 
&= 
\int_t^T  \frac \partial{\partial s} \langle \varphi ,  Jv^s \rangle \, ds \ \\
&= \frac 1 \logeps \int_t^T \int_\Omega 
\left( \ep_{lj} \varphi_{x_l}
{\eta^2_{x_j}}
\frac
{(|v_{\eps}|^2-1)^2}{4\eps^2}\ +  \ep_{lj}  \varphi_{x_kx_l} v_{\eps,x_j}\cdot v_{\eps,x_k}\right)
\\
&\quad\quad\quad+  \int_t^T\int_\Omega 
\ep_{lj} \varphi_{x_l}
\frac{\eta^2_{x_k}}{\eta^2} \frac{ v_{\eps,x_j}\cdot v_{\eps,x_k} }\logeps.
\end{aligned}
\end{equation}
We immediately see from  \eqref{eq:estimparsurplusbis} 
that
\begin{equation}\label{eq:easyterm1}
\left| \frac 1 \logeps  \int_\Omega 
\left( \ep_{lj} \varphi_{x_l}
{\eta^2_{x_j}}
\frac {(|v|^2-1)^2}{4\eps^2}\  \right)\right|
\le C(\Sigma^0+ g(r_a^s)) 
\end{equation}
for every $s\in [t,T]$. 
Moreover, it follows from \eqref{small_t.3} that
$B(\xi_i(s);4 \logeps^{-1}) \subset B(a_i(t), \frac 12 \rho_{min})$
if $\eps_0$ is small enough, 
and the definition of $\varphi$ 
implies that $\varphi_{x_i x_j} = 0$ in $\cup_{i=1}^l  B(a_i(t), \frac 12 \rho_{min})$, so
\eqref{eq:usuelle} implies that
\begin{equation}\label{eq:easyterm2}
\begin{aligned}
\frac 1 \logeps  \int_\Omega 
\left|\ep_{lj} \varphi_{x_kx_l} v_{\eps,x_j}\cdot v_{\eps,x_k}\right|
&\le 
C \frac 1\logeps E_\eps(v; \Omega\setminus \cup B(\xi_i, 4\logeps^{-1}))\\
&\le C(\Sigma^0+ g(r_\xi^s)) \le C(\Sigma^0+ g(r_a^s)) 
\end{aligned}\end{equation}
for every $s\in [t,T]$.

{\bf 3}. We now decompose the remaining term in \eqref{eq:p5.3}. For every $s\in [t,T]$,
let\footnote{Note that the regularization scale $r_\xi^t$ is fixed for $s\in [t,T]$.} $j_*^s = j_*( \{\xi_i(s)\}, r_\xi^t)$  be the approximation to $j(v^s)$ obtained in
Proposition \ref{prop:approx}. Note that 
\[
v_{\eps,x_j}\cdot v_{\eps,x_k} =
|v_{\eps}| _{x_j}\, |v_{\eps}|_{x_k}
+
\frac {j(v)_j}{|v|} \ \frac {j(v)_k}{|v|} 
\]
where for example $j(v)_j = (iv, \partial_{x_j}v)$ denotes the $j$th component of $j(v)$.
Thus,  adding and subtracting $j_*^s$ in various places, and writing $\psi_{jk}$ as an abbreviation for 
$\ep_{lj} \varphi_{x_l}\frac{\eta^2_{x_k}}{\eta^2}$, 
we have for every $s\in [t,T]$,
\begin{equation}\label{eq:decompose}
\begin{aligned}
&\int_\Omega 
\psi_{jk} 
\frac{ v_{\eps,x_j}\cdot v_{\eps,x_k} }\logeps
=
\int_\Omega 
\psi_{jk} 
\frac {( j_*^s)_j ( j_*^s)_k}\logeps \ \ \\
&\quad\quad\quad + 
\int_\Omega 
\psi_{jk} 
\frac 1 \logeps \left[
\left( \frac {j(v)}{|v|} - j_*^s\right)_j (j_*^s)_k +
\left( \frac {j(v)}{|v|} - j_*^s\right)_k (j^s_*)_j\right] \ \\
&\quad\quad\quad+
\int_\Omega 
\psi_{jk} 
\frac 1 \logeps \left[  {|v_{\eps}| _{x_j}\, |v_{\eps}|_{x_k}}
+
\left( \frac {j(v)}{|v|} - j_*^s\right)_j
\left( \frac {j(v)}{|v|} - j_*^s\right)_k\right] \ .
\end{aligned}
\end{equation}
We immediately dispense with the easiest terms by using \eqref{eq:estimparsurplusbis} 
to see that
\begin{equation}\label{eq:iniesta}
\int_\Omega 
\psi_{jk} 
\frac 1 \logeps \left[  {|v_{\eps}| _{x_j}\, |v_{\eps}|_{x_k}}
+
\left( \frac {j(v)}{|v|} - j_*^s\right)_j
\left( \frac {j(v)}{|v|} - j_*^s\right)_k\right] \  \le C(\Sigma^0 + g(r_a^s))
\end{equation}
for every $s\in [t,T]$. 

{\bf 4}. We next consider the first term on the right-hand side of \eqref{eq:decompose}, which
is the term that  yields the dominant contribution.
Since $j_*^s$ is supported in $\cup_i B(\xi_i(s), \logeps^{-1})$, clearly
\[ 
\int_\Omega 
\psi_{jk}
\frac {( j_*^s)_j ( j_*^s)_k}\logeps \ \ 
=
\sum_{i=1}^l
\int_{B(\xi_i(s), \logeps^{-1})} \psi_{jk} \frac {( j_*^s)_j ( j_*^s)_k}\logeps \ .
\] 
For each $i=1,\ldots, l$, if $x\in B(\xi_i(s), \logeps^{-1})$, then
$|x-a_i(s)| \le \logeps^{-1} + r^s_a + r^s_\xi$,  by \eqref{eq:aminusxi}, 
so for every $s\in [t,T]$,
\begin{align*}
\left|\int_{B(\xi_i(s), \logeps^{-1})}\Bigl( \psi_{jk}(x) - \psi_{jk}(a_i(s))\Bigr)\frac {( j_*^s)_j ( j_*^s)_k}\logeps \ 
\right|
\ &\le \  \| \nabla \psi_{jk}\|_\infty(\logeps^{-1} + r^s_a + r^s_\xi) \frac{\| j_*^s\|_2^2}\logeps \\
\ &\le \ C (\logeps^{-1} + r^s_a + r^s_\xi) ,
\end{align*}
using the explicit form of $j_*^s$, which (together with the definition
\eqref{rxit.def} of $r_\xi^s$) also implies that
\begin{align*}
\int_{B(\xi_i(s), \logeps^{-1})}\frac {( j_*^s)_j ( j_*^s)_k}\logeps \ 
& = \ \frac{ \pi}\logeps  \delta_{jk} (\log \frac 1 {r_\xi^s}  - \log \logeps + \frac 14) \ \\
&= \ \pi\, \delta_{jk}\  \bigl( 1 - C_1(\Sigma^0 + g(r^s_a)) \bigr) \ + \  O(\frac{\log \logeps}{\logeps}).
\end{align*}
Combining the above computations  and recalling that $g(r)\ge  \max(r,  \frac{ \log \logeps}\logeps)$ for all $r$ and that
$g(r_a^s) \ge r_\xi^s$ for $s\le T_{loc}$, we conclude that
\begin{align}
\int_\Omega 
\psi_{jk}
\frac {( j_*^s)_j ( j_*^s)_k}\logeps \ \ 
&= \pi \sum_i \psi_{kk}(a_i(s)) + O(C_1(\Sigma^0 + g(r_a^s))\nonumber\\
&= \pi \frac d{ds}\left(\sum_i  d_i \varphi(a_i(s)) \right) + O(C_1(\Sigma^0 + g(r_a^s)).
\label{eq:mainterm}\end{align}
In the last line we have used the definition 
$\psi_{kk} = \ep_{lk}\varphi_{x_l} \partial_{x_k}(\log \eta^2 )= \nabla\varphi \cdot \nabla^\perp(\log \eta^2)$
together with the ordinary differential equation \eqref{eq:limitdynamics} satisfied by the points $a_i(\cdot)$. 

{\bf 5}. Combining \eqref{eq:p5.2}, \eqref{eq:p5.3}, \eqref{eq:easyterm1}, \eqref{eq:easyterm2}, \eqref{eq:decompose},  \eqref{eq:iniesta}, and \eqref{eq:mainterm}, 
and recalling \eqref{rsxi}, \eqref{gras}, we find that
\begin{align}\label{eq:halfway}
r^T_a - r^t_a
&\le   C(T-t) \bigl(\Sigma^0 + g(r_a^t) \bigr)\ \ + \ C r^t_\xi \ +
\  \int_t^T
\int_\Omega 
\frac {\psi_{jk} }\logeps
\left( \frac {j(v)}{|v|} - j_*^s\right)_j (j_*^s)_k \, dx\,ds\\
&
\quad\quad
\quad\quad
\quad\quad
\quad\quad
 +
 \int_t^T\int_\Omega 
\frac {\psi_{jk} } \logeps \left( \frac {j(v)}{|v|} - j_*^s\right)_k (j^s_*)_j\, dx\,ds  .\nonumber
\end{align}
We now begin to control the integrals on the right-hand side above. We will consider 
only the first, since the estimate of the second is identical.
First,
\begin{align}
 \int_t^T
\int_\Omega 
\frac {\psi_{jk} }\logeps
\left( \frac {j(v)}{|v|} - j_*^s\right)_j (j_*^s)_k \, dx\, ds
\ & = \  
\int_\Omega 
\frac {\psi_{jk} }\logeps (j_*^t)_k
 \int_t^T\left( \frac {j(v)}{|v|} - j_*^s\right)_j  \ ds \, dx\label{hardterm1}\\
\ &  + \ 
\int_\Omega 
\int_t^T \frac {\psi_{jk} }\logeps
(j_*^s - j_*^t)_k \left( \frac {j(v)}{|v|} - j_*^s\right)_j  \ ds \, dx \, . \nonumber
\end{align}
We claim that
\begin{equation}
\int_\Omega 
\int_t^T \frac {\psi_{jk} }\logeps
(j_*^s - j_*^t)_k \left( \frac {j(v)}{|v|} - j_*^s\right)_j  \ ds \, dx \, 
\le
C (T-t) (\Sigma^0  + g(r_a^t)).
\label{hardterm1a}\end{equation}
Using the Cauchy-Schwarz inequality, \eqref{eq:estimparsurplusbis}, and \eqref{gras}, we see that it suffices to prove that
\[
\int_\Omega |j_*^s - j_*^t|^2 \ dx \ \le \
 (\Sigma^0 + g(r_a^t)) \logeps
\quad\quad\quad\mbox{ for every $s\in [t,T]$. }
\]
Toward this end, we fix some such $s$, and   we introduce the notation
\[
\bar \xi_i := \frac 12 (\xi_i^t + \xi_i^s),\quad\quad
\sigma := 
r_\xi^s + r_\xi^t +\sum_{i=1}^l |\xi_i(t) - \xi_i(s)|.
\]
Our choice of $T$ and \eqref{small_t.1} imply that $\sigma \le 
 C\frac{(r_\xi^t)^2}\eps$.
 Writing $B_i := B(a_i(t), \frac{\rho_{min}}{2})$, we deduce from \eqref{small_t.3} and the support 
properties of $j_*$ that 
\[
\int_\Omega |j_*^t - j_*^s|^2 \ dx 
\ 
 = \sum_{i=1}^l 
 \int_{B_i}|j^t_* - j^s_*|^2 \ dx.
\]
For each $i$, $B(\bar \xi_i, \sigma) \subset B(\xi_i(t), 2\sigma)\cap B(\xi_i(s), 2\sigma)$, so
by an explicit computation, and recalling \eqref{rsxi} and the definition \eqref{rxit.def} of $r_\xi^t$,
we find that
\begin{align*}
\int_{B(\bar \xi_i, \sigma) }\!\!\!\!|j^t_* - j^s_*|^2 \ dx
\le 
2 \int_{B(\xi_i(t), 2 \sigma) }\!\!\!|j^t_* |^2 \ dx
+
2 \int_{B(\xi_i(s), 2 \sigma) }\!\!\!|j^s_* |^2 \ dx
&\le 
2\log( \frac{2\sigma}{r_\xi^t}) + 
2\log( \frac{2\sigma}{r_\xi^s}) + C 
\\
&\le 
C\log( \frac {r_\xi^t} \eps)\\
&\le
C (\Sigma^0 + g(r_a^t)) \logeps.
\end{align*}
Next, on $B(\bar \xi_i, \frac 1{2\logeps})$, the definitions imply that both $j^s$ and $j^t$ are nonzero, and
in fact
\[
|j^t_*(x) - j^s_*(x)|^2 = \frac {|\xi_i(t)-\xi_i(s)|^2}{|x - \xi(t)|^2 |x - \xi_i(s)|^2}.
\]
Since $|\xi_i(\tau) - \bar \xi_i| \le \frac \sigma 2$ for $\tau = t,s$,
it follows that 
\[
|j^t_*(x) - j^s_*(x)|^2 \ \ \le 
\frac {4 \sigma^2}{|x - \bar \xi_i|^4} \quad\mbox{ on }B(\bar \xi_i, \frac 1{2\logeps})\setminus B(\bar \xi_i, \sigma)
\]
and hence that 
\[
\int_{B(\bar \xi_i, \frac 1{2\logeps})\setminus B(\bar \xi_i, \sigma) }|j^t_* - j^s_*|^2 \ dx \ 
\le C.
\]
Finally, 
\[
\int_{B_i \setminus B(\bar \xi_i, \frac 1{2\logeps}) }|j^t_* - j^s_*|^2 \ dx \ 
\le
2\int_{B_i \setminus B(\xi_i(t), \frac 1{4\logeps}) }|j^t_* |^2 \ dx \ 
+
2\int_{B_i \setminus B(\xi_i(s), \frac 1{4\logeps}) }|j^s_* |^2 \ dx \ \ \le C.
\]
We deduce \eqref{hardterm1a} by combining the previous inequalities.\

{\bf 6}.
We now consider the first term on the right-hand side of \eqref{hardterm1}.
Clearly
\begin{align}
\int_\Omega 
\frac {\psi_{jk} }\logeps (j_*^t)_k
 \int_t^T\left( \frac {j(v)}{|v|} - j_*^s\right)_j  \, ds \, dx
\ & = \  
\sum_{i=1}^l  
 \int_{B_i}\int_t^T \frac {\psi_{jk} }\logeps (j_*^t)_k  
\left( \frac {j(v)}{|v|} - j(v)\right)_j ds \, dx \nonumber \\
&\quad\quad+ \ 
\sum_{i=1}^l \int_{B_i}\int_t^T \frac {\psi_{jk} }\logeps (j_*^t)_k 
\left(  {j(v)} - j_*^s\right)_j ds \, dx.
\label{hardterm2}\end{align}
By elementary estimates,
\[
| \frac{j(v)}{|v|} - j(v)| =  \frac{|j(v)|}{|v|} \, \Bigl| |v|-1\Bigr|
\le 
\frac \eps2 |\nabla v|^2 + \frac 1{2\eps}(|v|^2 -1)^2  ,
\]
and  from the definitions, and recalling \eqref{C0gtr1}, we see that 
$\| j_*^t\|_\infty \le (r_\xi^t)^{-1} \le (\eps\logeps)^{-1}$. 
Thus  for every $i$, 
\begin{equation}
\begin{aligned}
\left| \int_{B_i}\int_t^T \frac {\psi_{jk} }\logeps (j_*^t)_k \cdot 
\left( \frac {j(v)}{|v|} - j(v)\right)
ds \, dx
\right|
&\ \le \  
C\| j_*^t\|_\infty \eps(T-t) \frac{E_{\eps,\eta}(v)}{\logeps}\\
&\le
\frac C \logeps(T-t))(\Sigma^0+C)\\
&\le
C(T-t)(\Sigma^0+ g(r_a^t)),
\end{aligned}
\label{hardterm3}\end{equation}
since 
\begin{equation}\label{totalE}
\frac {E_{\eps,\eta}(v)}{\logeps} \le
\Sigma^0 + \pi \sum_{i=1}^l \eta^2(a_i(t))
\le
\Sigma^0 + C(l, \|\eta\|_\infty).
\end{equation}

{\bf 7}. 
Now fix some $i\in \{1,\ldots, l\}$ and let $\tilde\chi^i\in C^\infty_c(B(a_i(t), \frac 34 \rho_{min}))$ be a function such that 
$\tilde\chi^i=1$ on $B_i$. Then for every $s\in [t,T]$,
\begin{equation}\label{eq:etaHodge}
\tilde \chi_i({j(v)} - j_*^s)
 = \nabla f^s +  \frac 1 {\eta^2}\nabla^\perp g^s \quad\quad\quad\mbox{ in } B(a_i(t), \frac 34 \rho_{min})
\end{equation}
for $f^s$ and $g^s$, real-valued functions on $B(a_i(t), \frac 34 \rho_{min})$,  solving
\begin{equation}\label{f.def}\begin{aligned}
\nabla\cdot(\eta^2\nabla f^s) &= \nabla\cdot\bigl(\tilde \chi^i \eta^2 (  j(v)  - j_*^s)\bigr) 
&\mbox{ in }B(a_i(t), \frac 34 \rho_{min}),\\
\nu \cdot\nabla f^s &= 0
&\mbox{ on }\partial B(a_i(t), \frac 34 \rho_{min}),
\end{aligned}
\end{equation}
and
\begin{equation}\label{g.def}\begin{aligned}
-\nabla\cdot( \frac{ \nabla g^s}{\eta^2}) &= \nabla\times \bigl(\tilde \chi^i( j(v)  - j_*^s)\bigr).
&\mbox{ in }B(a_i(t), \frac 34 \rho_{min}),\\
g^s &= 0
&\mbox{ on }\partial B(a_i(t), \frac 34 \rho_{min}).
\end{aligned}
\end{equation}
Indeed, if we let $f^s$ be a solution of \eqref{f.def}, then 
$\eta^2(\tilde \chi_i (j(v) - j_*^s) - \nabla f^s)$
is divergence-free and hence can be written as $\nabla^\perp g^s$ on $B(a_i(t), \frac 34 \rho_{min})$,
so that \eqref{eq:etaHodge} holds. Then it follows from \eqref{f.def} that $g^s$ satisfies the equation in \eqref{g.def}, and that the boundary condition
is satisfied after adding a constant to $g^s$.

Thus
\begin{equation}
\int_{B_i}\int_t^T \frac {\psi_{jk} }\logeps (j_*^t)_k \,
\left(  {j(v)} - j_*^s\right)_j ds \, dx
\ = \ 
\int_{B_i}\frac {\psi_{jk} }\logeps (j_*^t)_k\,  (\nabla F + \frac{\nabla^\perp G}{\eta^2})_j\, dx
\label{eq:moresplitting}\end{equation}
for
\[
F(x) = \int_t^T f^s(x) \ ds, 
\quad\quad\quad
G(x) = \int_t^T g^s(x) \ ds.
\]
We write 
$F = F_1+\cdots + F_4$, where
\[
\nabla\cdot(\eta^2 \nabla F_m) = A_m \ \ \mbox{ in }B(a_i(t), \frac 34 \rho_{min}),
\quad
\nu \cdot \nabla F_m = 0 \ \ \mbox{ in }\partial B(a_i(t), \frac 34 \rho_{min}),
\]
for
\begin{align*}
A_1 &= \tilde \chi^i  \int_t^T \nabla \cdot (\eta^2 j(v)) \ ds\\
A_2 &= - \tilde \chi^i  \int_t^T \nabla \cdot (\eta^2j_*^s) \ ds\\
A_3 &=   \int_t^T \eta^2  \nabla \tilde \chi^i\cdot \frac{j(v)}{|v|} (|v| - 1) \ ds\\
A_4 &=   \int_t^T \eta^2  \nabla \tilde \chi^i\cdot( \frac {j(v)}{|v|} - j_*^s) \ ds.
\end{align*}
Using the continuity equation  \eqref{eq:continuity} --- this is a key point in our
argument --- 
and \eqref{totalE},
we note that
\begin{align*}
\|A_1 \|_{L^2} & =  
\logeps \left\| \tilde \chi^i \eta^2  \left. (|v|^2-1) \right|_t^T  \right\|_{L^2}\\
&\le
C \eps \logeps\left( E_{\eps,\eta}(v^T) + E_{\eps, \eta}(v^t)\right)\\
&\le C \eps\logeps^3 (\Sigma^0 + g(r_a^t))
\end{align*}
since $g(r) \ge \frac {1+\log\logeps}{\logeps}$ for all $r$. 
Next, the definition implies that $\nabla \cdot j_*^s = 0$ for every $s$ and that 
$\| j_*^s\|_{L^p} \le C_p \logeps^{1 - \frac 2p}$ for every $p<2$, so 
\begin{align*}
\| A_2\|_{L^p} \le  \left\| \tilde \chi_i \ \right\|_{L^\infty} (T-t)
\sup_{s\in [t,T]}
\| \nabla(\eta^2)\cdot j_*^s  \|_{L^p} \le C (T-t) \logeps^{1-\frac 2p}\quad\mbox{for }p<2.
\end{align*}
Very much as in  \eqref{hardterm3}, we can check that
\[
\|A_3\|_{L^1} \ \le \ C (T-t) \sup_{s\in[t,T]} \| \frac {j(v)}{|v|}(|v|-1)\|_{L^1}
\le C (T-t)  \eps \logeps^2 (\Sigma^0 + g(r^t_a)),
\]
and it follows from \eqref{eq:estimparsurplusbis} and \eqref{gras} that 
\[
\|A_4\|_{L^2} \ \le \ C (T-t) \Big( \logeps (\Sigma^0 + g(r^t_a)) \Big)^{1/2}.
\]
Clearly, for any $q_1,\ldots, q_4\in [1,\infty]$,
\[
\int_{B_i}\frac {\psi_{jk} }\logeps (j_*^t)_k \cdot (\nabla F)_j\, dx
\le  \frac C \logeps \sum_{m=1}^4 \| j_* \|_{q_m} \|\nabla F_m \|_{q_m'}
\]
where $\frac 1{q_m} + \frac 1{q_m'}=1$.
Using elliptic estimates and Sobolev embedding theorems, and taking
$q_1=\frac 43$,
\[
\frac 1\logeps  \| j_* \|_{\frac 4 3} \|\nabla F_1 \|_{4}
\le 
\frac{C}\logeps \| j_* \|_{\frac 4 3} \| A_1 \|_{2}
\ \le 
\ C
\eps \logeps^{\frac 3 2} (\Sigma^0 + g(r_a^t)) \le C(T-t) (\Sigma^0 + g(r_a^t)).
\]
The last inequality follows from the choice of $T$ and  \eqref{C0gtr1}, which imply in
particular that $T-t \ge \eps\logeps^2$.
Similarly, taking $q_4 = 4/3$,
\[
\frac 1\logeps  \| j_* \|_{\frac 4 3} \|\nabla F_4 \|_{4}
\le 
\frac C{\logeps^{\frac 3 2}} \| A_4 \|_{2}
\ \le 
C\frac{(T-t)} \logeps (\Sigma^0+ g(r^t_a))^{1/2}
\le  C(T-t)(\Sigma^0 + g(r_a^t)),
\]
since $\logeps^{-1} \le g(r^t_a)$.
For any $q_2\in (1,2)$, taking $p_2<2$ such that $p_2^* = q_2'$, so that $\frac 1{p_2} = \frac 32 - \frac 1{q_2}$, 
we find from our estimate of $A_2$ that
\[
\frac 1\logeps  \| j_* \|_{q_2} \|\nabla F_2 \|_{q_2'}
\le 
\frac C\logeps \| j_* \|_{q_2} \| A_2 \|_{p_2}
\ \le 
C(T-t) \logeps^{-2} \le C(T-t)g(r^t_a).
\]
And, recalling by Stampacchia's estimate that for any $p\in [1,2)$ there exists $C_p$ such that
$\| \nabla F_3\|_p \le C_p \|A_3\|_1$, we compute, 
choosing  $q_3=3$ for concreteness,
\begin{align*}
\frac 1\logeps  \| j_* \|_{3} \|\nabla F_3 \|_{\frac 3 2}
\le 
 \| j_* \|_{3} \| A_3 \|_{1}
&\ \le C  (T-t) (r_\xi^t)^{-\frac 13 } \eps \logeps (\Sigma^0 + g(r^t_a))
\\
&\le   (T-t)  ( \eps \logeps)^{2/3}  (\Sigma^0 + g(r^t_a))
\end{align*}
again using the fact that $r_\xi^t \ge \eps\logeps$ for all $t$, see \eqref{C0gtr1}.
Combining the above, we find that
for every $i\in \{1,\ldots,l\}$ and  $0<\eps<\eps_0$ with $\eps_0$   sufficiently small, 
\begin{align}
\int_{B_i}\frac {\psi_{jk} }\logeps (j_*^t)_k \cdot (\nabla F)_j\, dx
&\le 
C(T-t)  (\Sigma^0+ g(r_a^t)).  
\label{Fterms}
\end{align}

{\bf 8}. Next, 
\[
\int_{B_i}\frac {\psi_{jk} }\logeps (j_*^t)_k \cdot \frac{(\nabla^\perp G)_j}{\eta^2}\, dx
=
\int_{B_i}\frac {\psi_{jk} }{\eta^2\logeps} (j_*^t)_k \cdot \nabla^\perp (G_1+G_2+G_3)_j \, dx
\]
for $G_m$ solving
\[
-\nabla\cdot ( \frac{ \nabla G_m}{\eta^2}) = A_m'
\quad \mbox{ in }B(a_i(t), \frac 34 \rho_{min}),
\quad\quad\quad
g = 0
\mbox{ on }\partial B(a_i(t), \frac 34 \rho_{min}),
\]
with
\begin{align*}
A_1'
&:= 
\int_t^T \tilde \chi^i \nabla\times(j(v) - j_*^s) \ ds,\\
A_2'
& := 
\int_t^T \nabla^\perp \tilde \chi^i  \cdot j(v)(1 - \frac 1{|v|}) \ ds,\\
A_3'
&:= 
\int_t^T \nabla^\perp \tilde \chi^i  \cdot (\frac{j(v)}{|v|} - j_*^s) \ ds.
\end{align*}
The terms containing $G_2$ and $G_3$ are estimated exactly as the terms containing
$F_3$ and $F_4$ in Step 7 above, leading to
\[
\int_{B_i}\frac {\psi_{jk} }{\eta^2\logeps} (j_*^t)_k \  \nabla^\perp (G_2+G_3)_j\, dx
\
\le \ C(T-t)(\Sigma^0+ g(r_a^t)) . 
\]
For the remaining term, we invoke the interpolation inequality
\begin{equation}
\| A_1'\|_{W^{-1,p}} \le C \| A_1'\|_{W^{-1,1}}^\theta \| A_1'\|_{L^1}^{1-\theta}  \label{negative.interp}
\end{equation}
 for $p\in (1,2)$ and $\theta$ such that 
 $\frac 1p = \frac \theta 1 + \frac {1-\theta}2$ (see e.g. \cite{Triebel} Theorem 2.4.1 combined with Sobolev embedding theorem).
To estimate the $W^{-1,1}$ norm, we fix $\zeta \in C^\infty_c(\Omega)$, and we compute
\begin{align*}
\langle \zeta, A_1'\rangle
= 
\int_t^T\langle \tilde \chi^i \zeta ,  \nabla\times (j(v) - j_*^s) \rangle
&\le
\int_t^T \| \tilde \chi^i \zeta \|_{W^{1,\infty}} \| \nabla\times (j(v) - j_*^s) \|_{W^{-1,1}} \ ds
\\
&\le C(T-t) r_\xi^t \|\zeta\|_{W^{1,\infty}}  \ 
\end{align*}
using \eqref{eq:JminusJstar} and \eqref{rsxi}.
Thus
\begin{equation}\label{A1primeflat}
\| A_1'\|_{W^{-1,1}} \le C(T-t) r_\xi^t .
\end{equation}
Also, for every $s\in [t,T]$,
\[
\| \nabla \times (j(v)- j_*^s) \|_{L^1} 
\le 
\| 2 Jv \|_{L^1} + \| \nabla\times j_*^s\|_{L^1}
\le 
C E_{\eps,\eta}(v) + 2\pi l.
\]
Estimating $E_{\eps,\eta}$ as usual by $C\logeps(\Sigma^0 + g(r_a^t))$, integrating the last inequality from $t$ to $T$, and combining it with \eqref{A1primeflat} and \eqref{negative.interp},
we obtain
\[
\|A_1'\|_{W^{-1,p}} \le C (T-t)(r_\xi^t)^\theta \left( C\logeps(\Sigma^0 + g(r_a^t))\right)^{1-\theta}.
\]
Then using H\"older's inequality and (again) the fact that $\| j_*^s\|_{p'} \le C (r_\xi^s)^{ \frac 2{p'}-1}$ for $p'>2$,
\begin{align*}
\int_{B_i}\frac {\psi_{jk} }{\eta^2 \logeps} (j_*^t)_k \cdot \nabla^\perp (G_1)_j\, dx
&\le 
\frac C\logeps (T-t) (r_\xi^t)^{\theta + \frac 2 {p'}-1}\left( C\logeps(\Sigma^0 + g(r_a^t))\right)^{1-\theta}
\nonumber\\
&\le
C(T-t) \logeps^{-\theta} \left( \Sigma^0 + g(r_a^t)\right)^{1-\theta}
\nonumber\\
&\le
C(T-t)  \left( \Sigma^0 + g(r_a^t)\right),
\label{G1est}
\end{align*}
since it turns out that $\theta+ \frac 2{p'}-1 = 0$, and noting that $\logeps^{-1} \le g(r_a^t)$ for all $t$.
Assembling these estimates, we find that
\[
\int_{B_i}\frac {\psi_{jk} }\logeps (j_*^t)_k (\nabla^\perp G)_j\, dx
\le
C(T-t)  \left( \Sigma^0 + g(r_a^t)\right). 
\]
Now by combining this with \eqref{eq:halfway}, \eqref{hardterm1}, \eqref{hardterm1a}, \eqref{hardterm3}, \eqref{Fterms}, we finally obtain
\begin{equation}
\begin{aligned}
r^T_a - r^t_a \le 
C(T-t)  \left( \Sigma^0 + g(r_a^t)\right) .
\end{aligned}
\label{main.est}\end{equation}
\end{proof}

\section{Proof of Theorem \ref{thm:main}}

Our main result is a straightforward corollary of the discrepancy estimate
proved in the previous section.

\begin{proof}[Proof of Theorem \ref{thm:main}]

Let $Y$ denote the solution of the ordinary differential equation
$$
\dot Y(t) = C_0\Big(\Sigma^0 + g(Y(t))\Big),\qquad Y(0) = r_a^0, 
$$
where $g$ is the function defined in \eqref{eq:defg}, and   
let $\{ Y_n\}_{n=0}^\infty$ be a discrete approximation to $Y(\cdot)$ 
obtained via an Euler approximation implicit in the statement
of Proposition \ref{prop:discrineq}.  Thus, we define
\[
Y_0= r^0_a, \qquad Y_{n+1} = Y_n +(t_{n+1} - t_{n}) \,C_0 \left(\Sigma^0 + g(Y_n)\right),
\]
\[
t_{n+1}  := t_{n} + \frac{(r^{n}_\xi)^2}\eps
\]
where 
\[
r^n_\xi :=  r_\xi(\Sigma^0, Y_n) = \eps \exp( C_0(\Sigma^0+g(Y^n)))\logeps).
\]
Since the function $f(Y):= C_0 \left(\Sigma^0 + g(Y)\right) $ is convex, 
a forward Euler approximation to the solution of the equation $Y' = f(Y)$ is
always less than or equal to the actual solution, and it follows that
$
Y_n \le Y(t_n)$ for all $t$.
Then repeated application of Proposition \ref{prop:discrineq} shows that
\[
r_a^{t_n} \le Y_n \le Y(t_n)
\mbox{ for every $n$ such that $t_n \le T_{col}$ and $\Sigma^0 + g(Y_n) \le  \frac 1{4C_1}$.}
\]
Given an arbitrary $t\in (0,T_{col}]$ such that $\Sigma^0+ g(Y(t))\le \frac 1{4C_1}$,
there exists some $n$ such that $t\in [t_n, t_{n+1}]$ and $r_a^{t_n} \le Y(t_n)$.
Then by Proposition \ref{prop:jacspeed}, see in particular \eqref{eq:rxis}, as well as \eqref{rsxi},
\begin{equation}\label{eq:parotide0}
r_a^t \le r_a^{t_n} + C\Big(  ( t_{n+1}-t_n) + r_\xi^{t_n}\Big)   \le Y(t) + C\eps^{1/2},
\end{equation}
since the bound $\Sigma^0+ g(Y(t_n))\le \frac 1{4C_1}$  guarantees that $r_\xi^{t_n}\le \eps^{3/4}$
and hence that $ t_{n+1}-t_n \le \eps^{1/2}.$ It remains to bound the function $Y$ from above. For that purpose, 
we notice that since $g(y) \leq y + \log\logeps/\logeps$ for every $y\geq 0$, we have $Y(t) \leq \tilde Y(t)$ where $\tilde Y$ is the solution of the ordinary differential equation 
$$\dot{ \tilde Y}(t) = C_0\big( \Sigma^0 + \frac{\log\logeps}{\logeps} + \tilde Y(t)\big),\qquad \tilde Y(0) = r_a^0.$$
The solution of the latter is explicitly given by
$$
\tilde Y(t) = r_a^0 + \big( \Sigma^0 + r_a^0+ \frac{\log\logeps}{\logeps} \big)\big( e^{C_0t}-1\big),
$$
and the conclusion therefore follows from \eqref{eq:parotide0}, increasing the value of $C_0$ to the value of 
$C$ in \eqref{eq:parotide0} if necessary. 
\end{proof}



\section{Some properties of the ground state}\label{sect:groundstate}

In this section we briefly recall some facts about  minimizers of the functional\footnote{Note that we make no restriction on the dimension $N$ here.}
\begin{equation}\label{eq:EepsV}
\boE_{\eps,V}(u) = \int_{\R^N} \frac{ |\nabla u|^2}2 + \frac 1{2\eps^2}\left( V(x) {|u|^2} + \frac12 {|u|^4}\right)dx
\end{equation}
in the space 
\begin{equation}\label{eq:constraint}
\mathcal{H}_m := \{ u\in H^1(\R^N;\C) : \int_{\R^N} V |u|^2 < \infty, \int_{\R^N} |u|^2 = m \}
\end{equation}
where $V:\R^N\to [0,\infty)$ is a smooth function such that $V(x)\to \infty$ as $|x|\to \infty$, and 
$m>0$ is a parameter.

For every positive $\eps,m$, the existence of a function $\eta_{\eps, m}:\R^N\to (0,\infty)$ minimizing $\boE_{\eps,V}$ in $\mathcal{H}_m$ is standard, and follows easily from the growth of $V$ (which implies that the $L^2$
constraint is preserved for weak limits of sequences with equi-bounded energy) together with the strong maximum principle and the fact that 
$\boE_{\eps,V}(|u|)  \le \boE_{\eps,V}(u)$ for all $u$. 

In the introduction, we already introduced the unique number $\lambda_0$ such that
\[
\int_{\R^N} (\lambda_0-V)^+ dx = m, 
\]
and we have denoted by $\rhotf := (\lambda_0 - V)^+$ the Thomas-Fermi profile associated to $V$ and $m.$  We also note  $w := (\lambda_0-V)^-$. We will prove

\begin{prop}\label{prop:groundstate}
Let $\eta = \eta_{\eps. m} \in \mathcal{H}_m$ be a positive 
minimizer of $\boE_{\eps, V}$ in $\mathcal{H}_m$.

Then
\begin{equation}
\| \eta^2 - \rhotf \|_{L^2(\R^N)} \le C \eps^{2/3}.
\label{eq:L2eta}\end{equation}
Moreover, for any $K\subset\subset \Otf := \{ x\in \R^N : \rhotf(x)>0\}$,
there exists a constant $C = C(m,V,K)$ such that
\begin{equation}
\| \eta^2 - \rhotf \|_{L^\infty(K)} \le C \eps^{2/3},
\quad\quad
\|\nabla \eta^2  \|_{L^\infty(K)} \le C.
\label{eq:unifLipschitz}\end{equation}
\end{prop}

This is quite standard, and is proved for particular potentials $V$ in \cite{IM1} for
example. We include a complete proof, since the references we know all impose
slightly more restrictive conditions than we consider here (for example, symmetry conditions,
or the assumption that $\lambda_0$ is a regular value of $V$). 

\begin{proof}
It suffices to prove the result for $\eps \le\eps_0$, for some $\eps_0>0$.

{\bf 1}. First, as is standard, for $u\in \mathcal{H}_m$ we rewrite
\begin{align*}
\boE_{\eps,V}(u) 
&= 
\int_{\R^N} \Big[\frac {|\nabla u|^2}2 + \frac 1{4\eps^2}(|u|^2-\rhotf)^2  + \frac 1{2\eps^2} w|u|^2\Big]\ dx
\  + \ \frac 1{\eps^2}(\lambda_0 \frac m2 - \frac 14 \int_{\R^N} \rhotf^2)\\
&=:
\boE_{\eps, \rhotf}(u) + C_1(\eps, m).
\end{align*}
Thus, it is clear that a function minimizes $\boE_{\eps,V}$ in $\mathcal{H}_m$ if and only if it minimizes $\boE_{\eps,\rhotf}$ in $\mathcal{H}_m$. 

{\bf 2}.  Next we claim that 
\begin{equation} 
\inf_{\mathcal{H}_m}\boE_{\eps, \rhotf} \le  C \eps^{-2/3}.
\label{eq:upper}\end{equation}
Note that this immediately implies \eqref{eq:L2eta}.
We verify \eqref{eq:upper} by choosing $U_\eps :=  c_\eps f_\eps(\sqrt \rhotf)$, where
\[
f_\eps(s) = 
\begin{cases}\eps^{-\alpha}s^2&\mbox{ if }s\le \eps^\alpha\\
s&\mbox{ if }s\ge\eps^\alpha,
\end{cases}
\]
where $c_\eps$ is chosen so that $U_\eps\in \mathcal{H}_m$. 
Then straightforward estimates very much like those in \cite{IM1}, for example, show that 
$\boE_{\eps,\eta}(U_\eps) \le C (\eps^{-\alpha} + \eps^{2\alpha-2})$, and \eqref{eq:upper} follows by
taking $\alpha = 2/3$. (This  crude estimate has the advantage of holding for {\em every} $m >0$,  so that we do not require $\lambda_0$ to be a regular value of $V$. If $\lambda_0$ is a regular value, then a variant of the same construction shows that  $\inf_{\mathcal{H}_m} \boE_{\eps, \rhotf} \le  C \logeps$.)

{\bf 3}. Since $V- \lambda_0 = (V-\lambda_0)^+ - (V- \lambda_0 )^- = w-\rhotf$, we may write the variational equation satisfied by $\eta$ in the form
\[
-\Delta \eta + \frac 1{\eps^2}( \eta^2  -  \rhotf + w)\eta = \frac 1{\eps^2} (\lambda_\eps-\lambda_0)\eta,
\]
where $\frac 1{\eps^2}\lambda_{\eps}$ is a Lagrange multiplier. Multiplying by $\eta$ and integrating, and using the fact that $\eta\in \mathcal{H}_m$,  we find   that
\[
\frac m {\eps^2}(\lambda_\eps - \lambda_0) = \int_{\R^2} |\nabla \eta|^2 + \frac 1 {\eps^2}\left[ w \eta^2 + (\eta^2-\rhotf)^2 + (\eta^2-\rhotf )\rhotf\right].
\]
It follows that 
\begin{equation}\label{eq:lambdaeps0}
\frac m{\eps^2}(\lambda_\eps - \lambda_0) \le
4  \boE_{\eps ,\rhotf}(\eta) +  
\frac 1{\eps^2}\|\rhotf\|_{L^2(\R^N)} \| \eta^2-\rhotf\|_{L^2(\R^N)}
\le C \eps^{-4/3}
\end{equation}
by \eqref{eq:upper} and \eqref{eq:L2eta}.

{\bf 4}. Now let $\rho_{\TF,\eps} :=  (\lambda_\eps - V)^+$.
It follows from \eqref{eq:lambdaeps0} that
\begin{equation}
\|\rho_{\TF,\eps} - \rhotf\|_{L^\infty(\R^N)}  = |\lambda_\eps - \lambda_0| \le C \eps^{2/3},
\label{eq:lambdaeps}\end{equation}
so that $K\subset\subset \Omega_\eps := \{ x\in \R^N : \rho_{\TF,\eps}>0\}$ if $\eps>0$ is  sufficiently small, which we henceforth take to be the case.
Note also that
\begin{equation}
-\Delta \eta + \frac 1{\eps^2}(\eta^2 - \rho_{\TF,\eps})\eta = 0
\quad\quad\mbox{ in $\Omega_\eps$.
}
\label{eq:etasol}\end{equation}
Now fix some $r \le \frac 12\operatorname{dist}(K,\partial \Otf)$. 
In view of \eqref{eq:lambdaeps}, and since $V$ is $C^2$, there exists $a,k>0$ and $\eps_0>0$ such that 
\begin{equation}\label{ab}
\rho_{\TF,\eps}>a^2 \ \mbox{ and }\ \ |\Delta \sqrt \rho_{\TF,\eps}| \le k
\mbox{ whenever }0<\eps\le \eps_0.
\end{equation}
For any $x\in K$ and $b\in (0,a)$, define 
\[
\zeta_{x,b}(y) = \zeta(y)=  b( \frac{|y-x|^2}{r^2}-1)^2
\]
in $B(x,r)$. Then for $b\in (0,\frac a2)$,
\begin{equation}\label{eq:zetasub}
-\Delta \zeta  + \frac 1{\eps^2}(\zeta^2-\rho_{\TF,\eps})\zeta 
\  \le \  -\Delta \zeta - \frac{3a^2}{4\eps^2} \zeta   \  < \ 0
\quad\quad\mbox{ in } B(x,r)
\end{equation}
whenever $\eps$ is sufficiently small. It follows that $\eta\ge \zeta_{x,b}$ in $B(x,r)$ for
every $b\in (0,\frac a2)$, as otherwise we could find some $b_0\in (0,\frac a2)$ such that $\min_{B(x,r)} (\eta-\zeta_{x,b_0}) = 0$. Since $\eta>0$, the minimum would have to be attained in the interior of $B(x,r)$, and this is impossible in view of \eqref{eq:etasol} and \eqref{eq:zetasub}.

It follows that  
\begin{equation}
\eta(y)\ge \frac{9a}{32} =:  \alpha \mbox{ in }B(x, r/2).
\label{eq:etalbd}\end{equation}
Note also that $\| \eta \|_{L^\infty(\R^N)} \le \| \sqrt{\rho_{\TF,\eps}}\|_{L^\infty(\R^N)}$, since otherwise $\tilde\eta  :=\min( \eta ,  \| \sqrt{ \rho_{\TF,\eps}}\|_\infty)$
would  satisfy $\boE_{\eps,\rhotf}(\tilde \eta) < \boE_{\eps,\rhotf}(\eta)$, contradicting the minimality of $\eta$.

{\bf 5}. Now write $\theta :=\eta-\sqrt{\rho_{\TF,\eps}}$. Then
\[
- \Delta \theta + a_\eps(x) \theta = \Delta \sqrt {\rho_{\TF,\eps}}
\quad\quad\mbox{ for }a_\eps(x) = \frac 1{\eps^2}(\theta + 2\sqrt{\rho_{\TF,\eps}})(\theta+\sqrt{\rho_{\TF,\eps}}) \  \overset{\eqref{eq:etalbd}}\ge  \ \frac {\alpha^2}{\eps^2}
\]
in $B(x, r/2)$, and $|\theta|\le  2 \|\sqrt\rho_{\TF,\eps}\|_{L^\infty(\R^N)}$ 
on $B(x, r/2)$. Now for $y\in B(x,r/2)$
define
\[
\Theta_\eps (y) := 
\frac k{\alpha^2} \eps^2 + 2 \|\sqrt{\rho_{\TF,\eps}}\|_{L^\infty(\R^N)}
\exp\left[ \frac{\alpha}{r\eps}( \frac{|y-x|^2}2- \frac {r^2} 8 )\right]
\]
where $k$ is the bound for $\|\Delta \sqrt {\rho_{\TF,\eps}}\|_\infty$ 
found in \eqref{ab}.
Then $\Theta \ge \theta$ on $\partial B(x,r/2)$, and there exists $\eps_0>0$ such that 
\[
(-\Delta + a_\eps )\Theta \ \ge \ k \  \ge\  (-\Delta + a_\eps)\theta \mbox{\ \  in \,  $B(x,r/2)$,\qquad 
if $0<\eps< \eps_0$.}
\]
It follows that
$\Theta \ge \theta$ in $B(x, r/2)$, and similarly $-\Theta \ge -\theta$ in $B(x, r/2)$.
Thus
\begin{equation}\label{eq:theta}
|\eta - \sqrt {\rho_\eps}| \le C \eps^2 \quad\mbox{ on }B(x, r/4).
\end{equation}

{\bf 6}.  Returning to \eqref{eq:etasol}, we see that 
\[
-\Delta \eta  + b_\eps \eta = 0 \quad\mbox{ in }B(x, r/4), \mbox{ for }b_\eps = \frac 1{\eps^2}(\eta^2-\rho_\eps),
\]
and \eqref{eq:theta} implies that $\|b_\eps\|_{L^\infty(B(x, r/4)} \le C$ independent of $\eps\in (0,\eps_0)$ and $x\in K$. Since we already know that $\|\eta\|_{L^\infty(\R^N)} \le C$, we conclude from standard elliptic regularity  that $\| \nabla \eta\|_{L^\infty(B(x, r/8))} \le C $. Also, it follows
from \eqref{eq:lambdaeps} and \eqref{eq:theta} that 
$\| \eta^2 - \rho \|_{L^\infty(K)} \le C \eps^{2/3}$, 
so we have proved \eqref{eq:unifLipschitz}.
\end{proof}

\section{Proof of Theorem \ref{thm:limit}}

In view of \eqref{eq:parotide2}, Theorem \ref{thm:limit} is a direct consequence of Theorem \ref{thm:main} combined with Proposition \ref{prop:groundstate} and the continuity of the solution of an initial value problem with respect to the nonlinearity.

\medskip

\noindent{\bf Addresses and E-mails: }

\smallskip
\noindent
Robert Jerrard. Department of Mathematics, University of Toronto, Toronto, Ontario M5S 2E4, Canada. E-mail:
{\tt rjerrard@math.utoronto.ca}

\smallskip

\noindent
Didier Smets. Laboratoire Jacques-Louis Lions, Universit\'e Pierre \& Marie Curie, 4 place Jussieu BC 187, 75252
Paris Cedex 05, France. E-mail: {\tt smets@ann.jussieu.fr}

\begin{thebibliography}{00}


\bibitem{AMSW} N. Ben Abdallah, F.  M\'ehats, C.  Schmeiser, and R.M. Weish\"aupl,
{\em The nonlinear Schr\"odinger equation with a strongly anisotropic harmonic potential},
SIAM J. Math. Anal. {\bf 37} (2005), 189--199. 
 
 
 \bibitem{BeJeSm} F. Bethuel, R.L. Jerrard, and D. Smets,
{\it On the NLS dynamics for infinite energy vortex configurations on the plane},
Rev. Mat. Iberoam. {\bf 24} (2008),  671--702. 

\bibitem{BCL}  H.~Brezis, J.M.~Coron, and E.H.~Lieb, 
{\it  Harmonic maps with defects},  Comm. Math. Phys., {\bf 107} (1986) 649--705. 

\bibitem{BrO} H. Brezis, L. Oswald, {Remarks on sublinear elliptic
	equations,} Nonlin. Anal. {\bf 10} (1986), 55--64.

\bibitem{CoJe} J.E. Colliander and R.L. Jerrard, {\it Vortex dynamics for the
 Ginzburg-Landau-Schr\"odinger equation,} Internat. Math. Res. Notices  {\bf 7}
(1998), 333--358.

\bibitem{CoJe2}
J.E. Colliander and R.L. Jerrard, {\it Ginzburg-Landau vortices: weak stability and Schr\"o\-din\-ger  equation dynamics}, J. Anal. Math. {\bf 77} (1999), 129--205.

\bibitem{IM1}R. Ignat, V. Millot, {\it The critical velocity for vortex existence
in a two-dimensional rotating Bose-Einstein condensate},   J. Funct.
Anal. {\bf 233} (2006), 260--306.



\bibitem{Je} R.L. Jerrard, {\it Lower bounds for generalized Ginzburg-Landau
functionals}, SIAM J. Math. Anal. {\bf 30} (1999), 721--746.

\bibitem{JeSo} R.L. Jerrard and H.M. Soner, {\it The Jacobian and the
Ginzburg-Landau energy,} Calc. Var. PDE {\bf 14}  (2002), 151--191.


 \bibitem{JeSp0} R.L. Jerrard and D.  Spirn, {\it  Refined Jacobian estimates
for Ginzburg-Landau functionals}, Indiana Univ. Math. Jour. {\bf 56} (2007),
135--186. 
 
\bibitem{JeSp} R.L. Jerrard and D.  Spirn, {\it  Refined Jacobian estimates and
Gross-Pitaevsky vortex dynamics}, Arch. Ration. Mech. Anal. {\bf 190} (2008), 425--475. 

\bibitem{JianSong} H-Y. Jian, B-H. Song, {\it Vortex dynamics of Ginzburg-Landau equations in inhomogeneous superconductors}, J. Differential Equations {\bf 170} (2001), 123--141.

\bibitem{KMMS} M. Kurzke, C. Melcher, R. Moser, and D. Spirn, {\it Dynamics for Ginzburg-Landau vortices under a mixed flow}. Indiana Univ. Math. J. {\bf 58} (2009), 2597--2621.

\bibitem{Li} F. H. Lin, {\it Complex Ginzburg-Landau equations and dynamics of vortices, filaments, and codimension-2 submanifolds}, Comm. Pure Appl. Math. {\bf 51} (1998), 385--441.

\bibitem{LiXi} F.H. Lin and J. Xin, {\it On the Incompressible Fluid Limit and the Vortex Motion Law of the Nonlinear Schr\"odinger Equation}, Comm. Math. Phys. {\bf 52} (1999), 249--274.

\bibitem{Miot} E. Miot, {\it Dynamics of vortices for the complex Ginzburg-Landau equation}, Anal PDE {\bf 2} (2009), 159--186.


\bibitem{m-z}A. Montero, {\it Hodge decomposition with degenerate weights and the Gross-Pitaevskii energy},
J. Funct. Anal. {\bf 254} (2008),  1926--1973.

\bibitem{Sa} E. Sandier, {\it Lower bounds for the energy of unit vector fields
and applications} J. Funct. Anal. {\bf 152} (1998), 379--403.

\bibitem{SerfTice}S. Serfaty and I. Tice, {\it Ginzburg-Landau vortex dynamics with pinning and strong applied currents}, Arch. Rational Mech. Anal. {\bf 201} (2011), 413--464. 

\bibitem{Triebel} H. Triebel, Interpolation theory, function spaces, differential operators. North-Holland Publishing Company, Amsterdam, New York, Oxford, 1978.

\end{thebibliography}
 \end{document}